\documentclass[12pt,reqno]{amsart}

\usepackage{eucal,color,graphicx,subfigure,mathrsfs,amssymb,stmaryrd}

\numberwithin{equation}{section}
\newcommand{\nn}{\nonumber}

\newtheorem{theorem}{Theorem}[section]
\newtheorem{lemma}[theorem]{Lemma}

\theoremstyle{definition}
\newtheorem{definition}[theorem]{Definition}

\theoremstyle{remark}

\newcommand{\lb}{\llbracket}
\newcommand{\rb}{\rrbracket}
\newcommand{\Lb}{\{\hspace{-4.0pt}\{}
\newcommand{\Rb}{\}\hspace{-4.0pt}\}}

\numberwithin{equation}{section}

\def\bfnu{\boldsymbol{\nu}}

\def\calF{\mathcal{F}}
\def\calT{\mathcal{T}}

\begin{document}

\title{A $C^{0}$ interior penalty method for $m$th-Laplace equation}

\author{Huangxin Chen}
\address{School of Mathematical Sciences and Fujian Provincial Key Laboratory on Mathematical Modeling and 
High Performance Scientific Computing, Xiamen University, Fujian, 361005, China}
\email{chx@xmu.edu.cn}

\author{Jingzhi Li}
\address{Department of Mathematics, Southern University of
Science and Technology, Shenzhen, 518055, China.}
\email{li.jz@sustech.edu.cn}

\author{Weifeng Qiu}
\address{Department of Mathematics, City University of Hong Kong,
83 Tat Chee Avenue, Kowloon, Hong Kong, China}
\email{weifeqiu@cityu.edu.hk}

%\thanks{Corresponding author: Weifeng Qiu.}

\thanks{The work of Huangxin Chen is 
supported by the NSF of China (Grant No. 12122115, 11771363). 
The work of Jingzhi Li  was partially supported by the NSF of China No. 11971221, Guangdong NSF Major Fund No. 2021ZDZX1001, 
the Shenzhen Sci-Tech Fund No. RCJC20200714114556020, JCYJ20200109115422828 and JCYJ20190809150413261,  
and Guangdong Provincial Key Laboratory of Computational Science
and Material Design No. 2019B030301001.
Weifeng Qiu's research is partially supported by the Research Grants Council of the Hong Kong Special Administrative Region, China. (Project No. CityU 11302219, CityU 11300621). The third author is corresponding author.} 

\begin{abstract}
In this paper, we propose a $C^{0}$ interior penalty method for $m$th-Laplace equation 
on bounded Lipschitz polyhedral domain in $\mathbb{R}^{d}$, where $m$ and $d$ can be any 
positive integers. The standard $H^{1}$-conforming piecewise $r$-th order polynomial space 
is used to approximate the exact solution $u$, where $r$ can be any integer greater than 
or equal to $m$. Unlike the interior penalty method in 
[T.~Gudi and M.~Neilan, {\em An interior penalty method for a sixth-order elliptic equation}, 
IMA J. Numer. Anal., \textbf{31(4)} (2011), pp. 1734--1753], we avoid computing $D^{m}$ 
of numerical solution on each element and high order normal derivatives of numerical 
solution along mesh interfaces. Therefore our method can be easily implemented. 
After proving discrete $H^{m}$-norm  bounded by the natural energy semi-norm associated with 
our method, we manage to obtain stability and optimal convergence with 
respect to discrete $H^{m}$-norm. {The error estimate under the low regularity assumption of the exact solution is also obtained.} Numerical experiments validate our theoretical estimate.
\end{abstract}



\maketitle

\section{Introduction}

We consider the $m$th-Laplace equation
\begin{subequations}
\label{m_Laplace_eqs}
\begin{align}
\label{m_Laplace_eq1}
& (-1)^{m} \Delta^{m} u =  f \quad \text{ in } \quad \Omega, \\
\label{m_Laplace_eq2}
& u =  \dfrac{\partial u}{\partial \bfnu} = \cdots = \dfrac{\partial^{m-1} u}{\partial \bfnu^{m-1}} = 0 
\quad \text{ on } \quad \partial \Omega,
\end{align}
\end{subequations}
where $m$ is an arbitrary positive integer, $\Omega$ is a bounded Lipschitz polyhedral domain in $\mathbb{R}^{d}$ 
($d=1,2,3,\cdots$), and $\bfnu$ is the outward unit normal vector field along $\partial \Omega$. The source term 
$f \in H^{-1}(\Omega)$.

Several works have been done to solve numerically (\ref{m_Laplace_eqs}). Standard $H^m$ conforming finite elements space requires $C^{m-1}$ continuity and leads to complicated construction of finite element space and lots of degrees of freedom when $m$ is large. Bramble and Zl\'amal \cite{Bramble1970} studied the $H^m$ conforming finite elements space on the two dimensional triangular meshes. Meanwhile, a $H^{m}$ conforming finite element space is developed by Hu and Zhang on rectangular grids 
for arbitrary $d$ in \cite{HuZhang2015}. Recently,  Hu, Lin and Wu introduce a $H^{m}$-conforming finite element space 
on simplicial meshes 
for any $d$ in \cite{HuLinWu2021}. The finite element space in \cite{HuLinWu2021} contains piecewise $r$-th order polynomials 
with $r \geq 2^{d}m + 1$. Therefore, the polynomial order of finite element space in \cite{HuLinWu2021} is quite big. 
Though up to this moment they have above mentioned restrictions, conforming $H^{m}$ finite element spaces 
are desirable in both theoretical analysis and practice. In order to simplify the construction of $H^m$ finite element 
space, alternative $H^{m}$ nonconforming finite element space is introduced in several works. In \cite{WangXu2013}, a $H^{m}$ nonconforming finite element space (named Morley-Wang-Xu elements) is introduced for $m \leq d$. Besides, Hu and Zhang also considered the $H^m$ nonconforming finite element space in \cite{HuZhang2017} on triangular meshes for $d=2$. The finite element space in \cite{WangXu2013} is generalized for $m = d +1$ by Wu and Xu in \cite{WuXu1}. Recently in \cite{WuXu2}, it is further generalized for arbitrary $m$ and $d$ but with stabilization along mesh interface in order to balance the weak continuity and the penalty terms.  In order to obtain stability and optimal convergence in some discrete $H^{m}$-norm, 
\cite{WangXu2013, WuXu1, WuXu2} 
propose to compute numerical approximation to $D^{m}u$, such that their implementation may become quite complicated 
as $m$ is large. The finite element spaces in \cite{Bramble1970, HuZhang2015} can be used to solve numerically 
(\ref{m_Laplace_eqs}) with any source term $f \in H^{-m}(\Omega)$. However, the implementation of these conforming and nonconforming finite element spaces can be quite challenging for large $m$. Virtual element methods have been investigated for (\ref{m_Laplace_eqs}). 
In \cite{Antonietti2020}, a conforming $H^{m}$ virtual element method is introduced for convex polygonal domain 
in $\mathbb{R}^{2}$. The finite element space in \cite{Antonietti2020} contains piecewise $r$-th order polynomials,  
where $r \geq 2m -1$. The virtual element method in \cite{Antonietti2020} needs strong assumption on 
regularity of $f$ ($f \in H^{r - m +1}(\Omega)$) to 
achieve optimal convergence (see \cite[Theorem~$4.2$]{Antonietti2020}).  In \cite{ChenHuang2020}, a nonconforming 
$H^{m}$ virtual element method is developed for bounded Lipschitz polyhedral domain in $\mathbb{R}^{d}$, where $d$ can 
be any positive integer. The design of finite element space in \cite{ChenHuang2020}, which contains piecewise $r$-th 
($r \geq m$) order polynomials, is based on a generalized Green's identity for $H^{m}$ inner product. 
It is assumed that $m \leq d$ in \cite{ChenHuang2020}. In \cite{Huang2020}, the virtual element method in 
\cite{ChenHuang2020} is extended for $m > d$. Besides above numerical methods based on 
primary formulations  of (\ref{m_Laplace_eqs}), a mixed formulation based on Helmholtz decomposition for tensor valued 
function is introduced in \cite{Mira2016} for two dimensional domain.

We propose a $C^{0}$ interior penalty method (\ref{IP_method}) for (\ref{m_Laplace_eqs}) for arbitrary positive 
integers $m$ and $d$. The finite element space of (\ref{IP_method}) is the standard $H^{1}$-conforming piecewise 
$r$-th order polynomials, where $r \geq m$. The design of (\ref{IP_method}) avoids computing $D^{m}$ 
of numerical solution on each element and high order normal derivatives of numerical solution along mesh interfaces. 
In fact, (\ref{IP_method}) only gets involved with calculation of high order multiplicity of Laplace 
of numerical solution ($\Delta^{i} u_{h}$ for $1\leq i \leq m$) and the gradient of high order multiplicity of 
Laplace of numerical solution ($\nabla \Delta^{i}u_{h}$ for $0 \leq i \leq m-1$) on both elements and mesh interfaces.
Therefore our method (\ref{IP_method}) can be easily implemented, even when $m$ is large and $d = 3$. 
After proving (Theorem~\ref{thm_stab}) that discrete $H^{m}$-norm (see Definition~\ref{def_discrete_H_m}) is bounded by 
the natural energy semi-norm associated with (\ref{IP_method}), we manage to show our method (\ref{IP_method}) 
has stability and optimal convergence on bounded Lipschitz polyhedral domain in $\mathbb{R}^{d}$  
with respect to the discrete $H^{m}$-norm, for any positive integers $m$ and $d$. Roughly speaking, we have 
\begin{align*}
\Vert u_{h}\Vert_{m, h} \leq & C \Vert f\Vert_{H^{-1}(\Omega)}, \\
\Vert u_{h} - u\Vert_{m, h} \leq & C h^{\min (r + 1 - m, s - m)} \Vert u \Vert_{H^{s}(\Omega)},
\end{align*}
where $s \geq 2m - 1$. We refer to Theorem~\ref{thm_wellposed} and Theorem~\ref{thm_conv} for detailed descriptions on stability 
and optimal convergence. The design and analysis of our method (\ref{IP_method}) can be easily generalized 
for nonlinear partial differential equations with $(-1)^{m}\Delta^{m}u$ as their leading term.  
We would like to point out that our method (\ref{IP_method}) is not a generalization 
of the interior penalty method for sixth-order elliptic equation ($m=3$) in \cite[($3.4$, $3.5$)]{GudiNeilan2011}. 
Actually, the method in \cite{GudiNeilan2011} needs to calculate numerical approximation to $D^{3}u$. 

{ If the exact solution of (\ref{m_Laplace_eqs}) is under the low regularity assumption, Gudi et al. have applied the analysis technique from the {\em a posteriori} error analysis to derive the error estimates for the interior penalty methods for the 2nd-order, 4th-order and 6th-order elliptic equations under the low regularity assumptions in \cite{Gudi2010,GudiNeilan2011}. In this paper, we shall extend the analysis by Gudi et al. for the proposed $C^0$ interior penalty methods for (\ref{m_Laplace_eqs}) when $m\geq 2$. Assuming $u \in H^m_0(\Omega)$ for (\ref{m_Laplace_eqs}), we have
\[
\Vert u - u_{h}\Vert_{m, h} \leq C   \inf_{v\in V_h}  \left( \| u-v\|_{m,h} + {\rm osc}_m(f)  \right),
\]
where $V_h$ is the $C^0$ conforming finite element space and ${\rm osc}_m(f)$ is the oscillation term defined in (\ref{def_osc_f}).
}

The numerical method considered in this paper works for any positive integers $m$ and $d$ from theoretical viewpoint. It can be applied to the practical high order equations. For instance, the modeling for plates in linear elasticity results in consideration of fourth-order partial differential equations \cite{Engel02}. Modeling in material science usually applies the fourth-order equation such as the Cahn-Hilliard equation \cite{ch1958,Elliott1986} and the sixth-order equation such as the thin-film equations \cite{Barrett2004} and the phase field crystal model \cite{Cheng2008,Wise2009,Wang_Wise2011}. Recently, an eighth-order equation was considered for the nonlinear Schr\"odinger equation in \cite{Kudryashov2020}. As mentioned in \cite{WangXu2013}, although there are rare practical applications for general higher order equations, the elliptic equations of order $m = d/2$ in any dimension have been used in differential geometry \cite{Chang2001}.  One can also extends the numerical methods and analysis for the solution of nonlinear 
Hamilton-Jacobi-Bellman equation and other phase-field models.

In the next section, we present the $C^{0}$ interior penalty method. In section 3, we prove stability and optimal convergence 
with respect to discrete $H^{m}$-norm (see Definition~\ref{def_discrete_H_m}). {In section 4, we show the error estimates under the low regularity assumption of the exact solution.} In section 5, we provide numerical experiments.

\section{$C^{0}$ interior penalty method}\label{sec:C0IP}
In this section, firstly we give notations to define the $C^{0}$ interior penalty method for (\ref{m_Laplace_eqs}). 
Then in section~\ref{sec_derivation_method}, we derive the $C^{0}$ interior penalty method for any $m \geq 1$. 
Finally in section~\ref{sec_examples_method}, we provide concrete examples of the method for $m=1,2,3,4$.

Let $\calT_{h}$ be a quasi-uniform conforming simplicial mesh of $\Omega$. Here we define $h = \max_{K \in 
\calT_{h}} h_{K}$ where $h_{K}$ is the diameter of the element $K\in \calT_{h}$. We denote by $\calF_{h}$, 
$\calF_{h}^{\text{int}}$ and $\calF_{h}^{\partial}$ the collections of all $(d-1)$-dimensional faces, 
interior faces and boundary faces of $\calT_{h}$, respectively. Obviously, $\calF_{h} = \calF_{h}^{\text{int}} 
\cup \calF_{h}^{\partial}$. For any positive integer $r$, 
we define {$V_{h} = H_{0}^{1}(\Omega) \cap P_{r}(\calT_{h})$}, where $P_{r}(\calT_{h}) 
= \{ v_{h} \in L^{2}(\Omega): v_{h}|_{K} \in P_{r}(K), \forall K \in \calT_{h} \}$. 

We introduce some trace operators. For any interior face $F \in \calF_{h}^{\text{int}}$, let $K^{-}, K^{+}\in \calT_{h}$ 
be two elements sharing $F$. We denote by $\bfnu^{-}$ and $\bfnu^{+}$ the outward unit normal vectors along $\partial K^{-}$ 
and $\partial K^{+}$, respectively. For scalar function $v:\Omega \rightarrow \mathbb{R}$ and vector field 
$\boldsymbol{\phi}: \Omega \rightarrow \mathbb{R}^{d}$, which may be discontinuous across $\calF_{h}^{\text{int}}$, 
we define the following quantities. For $v^{-}:= v |_{K^{-}}$, $v^{+}:= v |_{K^{+}}$, 
$\boldsymbol{\phi}^{-}:= \boldsymbol{\phi} |_{K^{-}}$ and 
$\boldsymbol{\phi}^{+}:= \boldsymbol{\phi} |_{K^{+}}$, we define 
\begin{align*}
& \Lb v \Rb = \dfrac{1}{2}\left( v^{-}|_{F} + v^{+}|_{F} \right), \quad \Lb \boldsymbol{\phi} \Rb 
= \dfrac{1}{2} \left( \boldsymbol{\phi}^{-}|_{F} + \boldsymbol{\phi}^{+}|_{F} \right),\\ 
& \lb v \rb = v^{-}\bfnu^{-}|_{F} + v^{+} \bfnu^{+}|_{F}, \quad \lb \boldsymbol{\phi} \rb = 
\boldsymbol{\phi}^{-}\cdot \bfnu^{-}|_{F} + \boldsymbol{\phi}^{+}\cdot \bfnu^{+}|_{F};  
\end{align*}    
if $F \in \partial K^{+} \cap \partial \Omega$, we define 
\begin{align*}
\Lb v \Rb = v^{+}|_{F}, \quad \Lb \boldsymbol{\phi} \Rb =\boldsymbol{\phi}^{+}|_{F} , \quad 
\lb v \rb = v^{+} \bfnu |_{F}, \quad \lb \boldsymbol{\phi} \rb = \boldsymbol{\phi}^{+}\cdot \bfnu |_{F}.
\end{align*}
{ We also define $\lb v\rb_J|_F = v^- |_F - v^+|_F$ for $F \in \calF^{\rm int}_h$ and $\lb v\rb_J|_F = v^+|_F$ for $F \in \partial \Omega$.}

\subsection{Derivation of $C^{0}$ interior penalty method}
\label{sec_derivation_method}  
We assume the exact solution $u \in H^{2m-1}(\Omega)$. For any $v_{h} \in V_{h}$, 
via $m$-times integrating by parts, 
\begin{align*}
& ((-1)^{m}\Delta^{m}u, v_{h})_{\Omega} \\
= &
\begin{cases}
& (\Delta^{\tilde{m}}u, \Delta^{\tilde{m}} v_{h})_{\calT_{h}} - \sum_{i=0}^{\tilde{m} - 1} 
\langle  \Delta^{\tilde{m}+i}u, \partial_{\bfnu} \Delta^{\tilde{m}- i - 1}v_{h} \rangle_{\partial \calT_{h}} 
+ \sum_{i=0}^{\tilde{m} - 2} \langle \partial_{\bfnu} \Delta^{\tilde{m}+i} u, 
\Delta^{\tilde{m} - i-1} v_{h} \rangle_{\partial \calT_{h}}, \\
& \qquad \text{ if } m = 2\tilde{m} \quad (m \text{ is an even number}); \\
& \\
%%%&(\Delta^{\tilde{m}}u, \Delta^{\tilde{m}} v_{h})_{\calT_{h}}  - \langle \Delta^{m-1}u, \partial_{\bfnu} v_{h} 
%%%\rangle_{\partial \calT_{h}}  + \langle \partial_{\bfnu} \Delta^{m-2}u,  \Delta v_{h} \rangle_{\partial \calT_{h}}
%%% - \langle \Delta^{m-2}u, \partial_{\bfnu} \Delta v_{h} \rangle_{\partial \calT_{h}} + \cdots\\
%%%& \qquad + \langle \partial_{\bfnu} \Delta^{\tilde{m}} u, \Delta^{\tilde{m}-1} v_{h} \rangle_{\partial \calT_{h}} 
%%%- \langle \Delta^{\tilde{m}} u, \partial_{\bfnu}  \Delta^{\tilde{m}-1} v_{h} \rangle_{\partial \calT_{h}}, 
%%%\quad \text{ if } m = 2\tilde{m} \quad (m \text{ is an even number}); \\
%%%& \\
& (\nabla \Delta^{\tilde{m}}u, \nabla \Delta^{\tilde{m}} v_{h})_{\calT_{h}}  
+ \sum_{i=0}^{\tilde{m} - 1} \langle \Delta^{\tilde{m}+i+1} u, 
\partial_{\bfnu} \Delta^{\tilde{m}- i -1} v_{h}  \rangle_{\partial \calT_{h}} 
 - \sum_{i=0}^{\tilde{m} -1} \langle \partial_{\bfnu} \Delta^{\tilde{m} + i}u, 
 \Delta^{\tilde{m} - i} v_{h} \rangle_{\partial \calT_{h}}, \\
& \qquad  \text{ if } m = 2\tilde{m}+1 \quad (m \text{ is an odd number}).
%%%&\\
%%%&(\nabla \Delta^{\tilde{m}}u, \nabla \Delta^{\tilde{m}} v_{h})_{\calT_{h}} 
%%%+ \langle \Delta^{m-1}u, \partial_{\bfnu} v_{h} \rangle_{\partial \calT_{h}} 
%%%- \langle \partial_{\bfnu} \Delta^{m-2}u,  \Delta v_{h} \rangle_{\partial \calT_{h}}
%%%+ \langle \Delta^{m-2}u, \partial_{\bfnu} \Delta v_{h} \rangle_{\partial \calT_{h}} + \cdots\\
%%%& \qquad - \langle \partial_{\bfnu} \Delta^{\tilde{m}+1} u, \Delta^{\tilde{m}-1} v_{h} \rangle_{\partial \calT_{h}} 
%%%+ \langle \Delta^{\tilde{m}+1} u, \partial_{\bfnu}  \Delta^{\tilde{m}-1} v_{h} \rangle_{\partial \calT_{h}}
%%%- \langle \partial_{\bfnu} \Delta^{\tilde{m}} u, \Delta^{\tilde{m}} v_{h} \rangle_{\partial \calT_{h}} , \\
%%%& \qquad \text{ if } m = 2\tilde{m}+1 \quad (m \text{ is an odd number}).
\end{cases}
\end{align*}  
Since $u \in H^{2m-1}(\Omega)$, for any $v_{h} \in V_{h}$, 
\begin{align}
\label{u_eq1}
& ((-1)^{m}\Delta^{m}u, v_{h})_{\Omega} \\
= &
\begin{cases}
& (\Delta^{\tilde{m}}u, \Delta^{\tilde{m}} v_{h})_{\calT_{h}} 
- \sum_{i=0}^{\tilde{m} -1} \langle \Lb \Delta^{\tilde{m} + i} u \Rb, 
\lb \nabla \Delta^{\tilde{m} - i -1}v_{h} \rb \rangle_{\calF_{h}} 
+ \sum_{i=0}^{\tilde{m} - 2} \langle \Lb \nabla \Delta^{\tilde{m} + i} u \Rb, 
 \lb \Delta^{\tilde{m} - i-1} v_{h}\rb  \rangle_{\calF_{h}}, \\
& \qquad  \text{ if } m = 2\tilde{m} \quad (m \text{ is an even number}); \\
& \\
%%%&(\Delta^{\tilde{m}}u, \Delta^{\tilde{m}} v_{h})_{\calT_{h}}  
%%%- \langle \Lb \Delta^{m-1}u \Rb, \lb \nabla v_{h} \rb
%%%\rangle_{\calF_{h}}
%%%+ \langle \Lb \nabla \Delta^{m-2}u \Rb, \lb \Delta v_{h} \rb \rangle_{\calF_{h}}
%%% - \langle \Lb \Delta^{m-2}u \Rb, \lb \nabla \Delta v_{h} \rb \rangle_{\calF_{h}} \\
%%%\nonumber
%%%& \qquad + \cdots + \langle \Lb \nabla \Delta^{\tilde{m}} u \Rb , \lb \Delta^{\tilde{m}-1} v_{h} \rb \rangle_{\calF_{h}} 
%%%- \langle \Lb \Delta^{\tilde{m}} u \Rb, \lb \nabla \Delta^{\tilde{m}-1} v_{h} \rb \rangle_{\calF_{h}}, 
%%%\quad \text{ if } m = 2\tilde{m}; \\
%%%& \\
& (\nabla \Delta^{\tilde{m}}u, \nabla \Delta^{\tilde{m}} v_{h})_{\calT_{h}}  
+ \sum_{i=0}^{\tilde{m} - 1} \langle \Lb \Delta^{\tilde{m} + i +1} u \Rb, 
\lb \nabla \Delta^{\tilde{m}-i-1} v_{h} \rb\rangle_{\calF_{h}} 
- \sum_{i=0}^{\tilde{m} - 1} \langle \Lb \nabla \Delta^{\tilde{m} + i} u\Rb, 
\lb \Delta^{\tilde{m} - i} v_{h} \rb \rangle_{\calF_{h}}, \\
\nonumber
& \qquad \text{ if } m = 2\tilde{m}+1 \quad (m \text{ is an odd number}). 
%%%& \\
%%%\nonumber
%%%&(\nabla \Delta^{\tilde{m}}u, \nabla \Delta^{\tilde{m}} v_{h})_{\calT_{h}} 
%%%+ \langle \Lb \Delta^{m-1}u \Rb,  \lb \nabla v_{h} \rb \rangle_{\calF_{h}}
%%%- \langle \Lb \nabla \Delta^{m-2}u \Rb, \lb \Delta v_{h} \rb \rangle_{\calF_{h}} \\
%%%\nonumber
%%%& \qquad  + \langle \Lb \Delta^{m-2}u \Rb, \lb \nabla \Delta v_{h} \rb \rangle_{\calF_{h}} 
%%%+ \cdots - \langle \Lb \nabla \Delta^{\tilde{m}+1} u \Rb, \lb \Delta^{\tilde{m}-1} v_{h} \rb 
%%%\rangle_{\calF_{h}} \\
%%%\nonumber
%%%& \qquad + \langle \Lb \Delta^{\tilde{m}+1} u \Rb, \lb \nabla \Delta^{\tilde{m}-1} v_{h} \rb \rangle_{\calF_{h}} 
%%% - \langle \Lb \nabla \Delta^{\tilde{m}} u \Rb, \lb \Delta^{\tilde{m}} v_{h} \rb \rangle_{\calF_{h}} , 
%%%\quad \text{ if } m = 2\tilde{m}+1.
\end{cases}
\end{align}

(\ref{u_eq1}) inspires us to define the coupling term $C_{h}$ in Definition~\ref{def_couple}.

\begin{definition}
\label{def_couple}
For any $w_{h}, v_{h} \in V_{h}$, we define the coupling term $C_{h}(w_{h}, v_{h})$ along mesh interface $\calF_{h}$ by 
\begin{align*}
& C_{h} (w_{h}, v_{h}) \\
= &
\begin{cases}
& - \sum_{i=0}^{\tilde{m} -1} \langle \Lb \Delta^{\tilde{m} + i} w_{h} \Rb, 
\lb \nabla \Delta^{\tilde{m} - i -1}v_{h} \rb \rangle_{\calF_{h}} 
+ \sum_{i=0}^{\tilde{m} - 2} \langle \Lb \nabla \Delta^{\tilde{m} + i} w_{h} \Rb, 
 \lb \Delta^{\tilde{m} - i-1} v_{h}\rb  \rangle_{\calF_{h}}, 
 \quad \text{ if } m = 2\tilde{m}; \\
& \\ 
& \sum_{i=0}^{\tilde{m} - 1} \langle \Lb \Delta^{\tilde{m} + i +1} w_{h} \Rb, 
\lb \nabla \Delta^{\tilde{m}-i-1} v_{h} \rb\rangle_{\calF_{h}} 
- \sum_{i=0}^{\tilde{m} - 1} \langle \Lb \nabla \Delta^{\tilde{m} + i} w_{h} \Rb, 
\lb \Delta^{\tilde{m} - i} v_{h} \rb \rangle_{\calF_{h}}, 
\quad \text{ if } m = 2\tilde{m}+1.
\end{cases}
\end{align*}
\end{definition}

In order to define $C^{0}$ interior penalty method, we need the stabilization term $S_{h}$ 
in Definition~\ref{def_stab_new}. 

\begin{definition}
\label{def_stab_new}
For any $w_{h}, v_{h} \in V_{h}$, we define the stabilization term $S_{h}(w_{h}, v_{h})$ 
along mesh interface $\calF_{h}$ by
\begin{align*}
 & S_{h} (w_{h}, v_{h}) \\  
 = &
\begin{cases}
& \sum_{i=0}^{\tilde{m} -1} h^{-(4i + 1)} \langle \lb \nabla \Delta^{\tilde{m} - i -1}w_{h} \rb, 
\lb \nabla \Delta^{\tilde{m} - i -1}v_{h} \rb \rangle_{\calF_{h}} 
+ \sum_{i=0}^{\tilde{m} - 2} h^{-(4i + 3)} \langle \lb \Delta^{\tilde{m} - i-1} w_{h}\rb, 
 \lb \Delta^{\tilde{m} - i-1} v_{h}\rb  \rangle_{\calF_{h}}, \\
& \qquad \text{ if } m = 2\tilde{m} \quad (m \text{ is an even number});\\
& \\
%%%& h^{-(2m-3)} \langle \lb \nabla w_{h} \rb, \lb \nabla v_{h} \rb  \rangle_{\calF_{h}}  
%%% + h^{-(2m-5)} \langle \lb \Delta w_{h} \rb, \lb \Delta v_{h} \rb \rangle_{\calF_{h}}  \\
%%%& \qquad + h^{-(2m-7)} \langle \lb  \nabla \Delta w_{h} \rb , \lb  \nabla \Delta v_{h} \rb 
%%%\rangle_{\calF_{h}} + h^{-(2m-9)} \langle \lb \Delta^{2} w_{h} \rb, \lb \Delta^{2} v_{h} \rb \rangle_{\calF_{h}} 
%%%+ \cdots \\
%%%& \qquad  + h^{-5} \langle \lb \nabla \Delta^{\tilde{m}-2}w_{h}\rb , \lb \nabla \Delta^{\tilde{m}-2}v_{h}\rb 
%%%\rangle_{\calF_{h}} + h^{-3} \langle \lb \Delta^{\tilde{m}-1}w_{h} \rb, \lb \Delta^{\tilde{m}-1}v_{h} \rb
%%%\rangle_{\calF_{h}} \\
%%%& \qquad + h^{-1} \langle \lb \nabla \Delta^{\tilde{m}-1}w_{h} \rb, \lb \nabla \Delta^{\tilde{m}-1}v_{h} \rb
%%%\rangle_{\calF_{h}}, \text{ if } m = 2\tilde{m}
%%%\quad (m \text{ is an even number}); \\
%%%& \\
& \sum_{i=0}^{\tilde{m} - 1} h^{-(4i+3)}\langle \lb \nabla \Delta^{\tilde{m}-i-1} w_{h} \rb, 
\lb \nabla \Delta^{\tilde{m}-i-1} v_{h} \rb \rangle_{\calF_{h}} 
+ \sum_{i=0}^{\tilde{m} - 1} h^{-(4i+1)} \langle \lb \Delta^{\tilde{m} - i} w_{h} \rb , 
\lb \Delta^{\tilde{m} - i} v_{h} \rb \rangle_{\calF_{h}}, \\
& \qquad \text{ if } m = 2\tilde{m} + 1 \quad (m \text{ is an odd number}).
%%%& \\
%%%& h^{-(2m-3)} \langle \lb \nabla w_{h} \rb, \lb \nabla v_{h} \rb \rangle_{\calF_{h}}
%%% + h^{-(2m-5)} \langle \lb \Delta w_{h} \rb, \lb \Delta v_{h} \rb \rangle_{\calF_{h}}  \\
%%%& \qquad + h^{-(2m-7)} \langle \lb  \nabla \Delta w_{h} \rb, \lb  \nabla \Delta v_{h} \rb 
%%%\rangle_{\calF_{h}} + h^{-(2m-9)} \langle \lb \Delta^{2} v_{h} \rb, \lb \Delta^{2} v_{h} \rb \rangle_{\calF_{h}}
%%%+ \cdots \\
%%%& \qquad  + h^{-5} \langle \lb \Delta^{\tilde{m}-1}w_{h} \rb, \lb \Delta^{\tilde{m}-1}v_{h} \rb
%%%\rangle_{\calF_{h}} + h^{-3}\langle \lb \nabla \Delta^{\tilde{m}-1}v_{h} \rb, \lb \nabla \Delta^{\tilde{m}-1}v_{h} \rb
%%%\rangle_{\calF_{h}} \\
%%%& \qquad + h^{-1} \langle \lb \Delta^{\tilde{m}}w_{h} \rb, \lb \Delta^{\tilde{m}}v_{h} \rb
%%%\rangle_{\calF_{h}}, \text{ if } m = 2\tilde{m} + 1
%%%\quad (m \text{ is an odd number}).
\end{cases}
\end{align*}
\end{definition}
We would like to point out that $S_{h} (w_{h}, v_{h}) =0$ if $m=1$.

The $C^{0}$ interior penalty method is to find $u_{h} \in V_{h}$, such that for any $v_{h}$,
{  
\begin{align}
\label{IP_method}
a_h (u_{h}, v_{h}) = (f, v_{h})_{\Omega},
\end{align}
where
\begin{align}
\label{IP_method_bilinear_form}
\begin{cases}
& a_h (u_{h}, v_{h}) = (\Delta^{\tilde{m}} u_{h}, \Delta^{\tilde{m}} v_{h})_{\calT_{h}} + C_{h}(u_{h}, v_{h}) + C_{h}(v_{h}, u_{h}) 
+ \tau S_{h}(u_{h}, v_{h}) , \text{ if } m = 2 \tilde{m}; \\
& \\
& a_h (u_{h}, v_{h}) = (\nabla \Delta^{\tilde{m}} u_{h}, \nabla \Delta^{\tilde{m}} v_{h})_{\calT_{h}} + C_{h}(u_{h}, v_{h}) + C_{h}(v_{h}, u_{h}) 
+ \tau S_{h}(u_{h}, v_{h}) , \text{ if } m = 2 \tilde{m} + 1.
\end{cases}
\end{align}
}
Here the parameter $\tau \geq 1$ shall be large enough but independent of $h$. 

\subsection{Examples of $C^{0}$ interior penalty method}
\label{sec_examples_method}

\begin{itemize}

\item $m=1$

The $C^{0}$ interior penalty method for $-\Delta u = f$ is to find $u_{h} \in V_{h}$ satisfying 
\begin{align}
\label{method_m_1}
(\nabla u_{h}, \nabla v_{h})_{\Omega} = (f, v_{h})_{\Omega}, \qquad \forall v_{h} \in V_{h}.
\end{align}

\item $m=2$ 

The $C^{0}$ interior penalty method for $\Delta^{2}u = f$ is to find $u_{h} \in V_{h}$ satisfying 
\begin{align}
\label{method_m_2}
& (\Delta u_{h}, \Delta v_{h})_{\calT_{h}} 
- \langle \Lb \Delta u_{h} \Rb, \lb \nabla v_{h} \rb\rangle_{\calF_{h}}
- \langle \Lb \Delta v_{h} \Rb, \lb \nabla u_{h} \rb\rangle_{\calF_{h}} \\
\nonumber
& \qquad + \tau h^{-1}\langle \lb \nabla u_{h} \rb, \lb \nabla v_{h}\rb \rangle_{\calF_{h}} 
= (f, v_{h})_{\Omega}, \qquad \forall v_{h} \in V_{h}.
\end{align}
{Actually, (\ref{method_m_2}) is the $C^0$ interior penalty method which replaces the discontinuous finite element spaces in \cite{Mozolevski2007} with $C^0$ finite element space.}

\item $m=3$

The $C^{0}$ interior penalty method for $-\Delta^{3} u = f$ is to find $u_{h} \in V_{h}$ satisfying 
\begin{align}
\label{method_m_3}
& (\nabla\Delta u_{h}, \nabla\Delta v_{h})_{\calT_{h}} 
+ \left( \langle \Lb \Delta^{2} u_{h} \Rb, \lb \nabla v_{h} \rb\rangle_{\calF_{h}}
- \langle \Lb \nabla \Delta u_{h} \Rb, \lb \Delta v_{h} \rb\rangle_{\calF_{h}} \right) \\
\nonumber 
&\qquad + \left( \langle \Lb \Delta^{2} v_{h} \Rb, \lb \nabla u_{h} \rb\rangle_{\calF_{h}}
- \langle \Lb \nabla \Delta v_{h} \Rb, \lb \Delta u_{h} \rb\rangle_{\calF_{h}} \right) \\
\nonumber
& \qquad + \tau \left( h^{-3}\langle \lb \nabla u_{h} \rb, \lb \nabla v_{h}\rb \rangle_{\calF_{h}} 
+ h^{-1} \langle \lb \Delta u_{h} \rb, \lb \Delta v_{h}\rb \rangle_{\calF_{h}} \right)
= (f, v_{h})_{\Omega}, \qquad \forall v_{h} \in V_{h}.
\end{align}

It is easy to see that (\ref{method_m_3}) is quite different from the interior penalty method in \cite{GudiNeilan2011}.

\item $m=4$

The $C^{0}$ interior penalty method for $\Delta^{4} u = f$ is to find $u_{h} \in V_{h}$ satisfying 
\begin{align}
\label{method_m_4}
& (\Delta^{2} u_{h}, \Delta^{2} v_{h})_{\calT_{h}} \\
\nonumber
& \qquad + \left( - \langle \Lb \Delta^{3} u_{h} \Rb, \lb \nabla v_{h} \rb\rangle_{\calF_{h}}
+  \langle \Lb \nabla \Delta^{2} u_{h} \Rb, \lb \Delta v_{h} \rb\rangle_{\calF_{h}} 
- \langle \Lb \Delta^{2} u_{h} \Rb, \lb \nabla \Delta v_{h} \rb\rangle_{\calF_{h}} \right) \\
\nonumber 
&\qquad + \left( - \langle \Lb \Delta^{3} v_{h} \Rb, \lb \nabla u_{h} \rb\rangle_{\calF_{h}}
+  \langle \Lb \nabla \Delta^{2} v_{h} \Rb, \lb \Delta u_{h} \rb\rangle_{\calF_{h}} 
- \langle \Lb \Delta^{2} v_{h} \Rb, \lb \nabla \Delta u_{h} \rb\rangle_{\calF_{h}} \right) \\
\nonumber
& \qquad + \tau \left( h^{-5}\langle \lb \nabla u_{h} \rb, \lb \nabla v_{h}\rb \rangle_{\calF_{h}} 
+ h^{-3} \langle \lb \Delta u_{h} \rb, \lb \Delta v_{h}\rb \rangle_{\calF_{h}}  
+ h^{-1} \langle \lb \nabla \Delta u_{h} \rb, \lb \nabla \Delta v_{h} \rb \rangle_{\calF_{h}} \right) \\
\nonumber
& \qquad \qquad = (f, v_{h})_{\Omega}, \qquad \forall v_{h} \in V_{h}.
\end{align}

\end{itemize}

\section{Analysis}\label{sec:analysis}

In this section, firstly we prove Theorem~\ref{thm_stab}, which states the discrete $H^{m}$-norm  
(see Definition~\ref{def_discrete_H_m}) bounded by the natural energy semi-norm associated with 
the $C^{0}$ interior penalty method (\ref{IP_method}). Then we prove Theorem~\ref{thm_wellposed}, 
which shows the energy estimate of (\ref{IP_method}). Finally, we prove 
Theorem~\ref{thm_conv}, which gives optimal convergence of numerical approximation to $u$ in the 
discrete $H^{m}$-norm. {Throughout
this paper, $C$ with or without a subscript denotes a positive constant depending only on the property of $\Omega$, the shape regularity of the meshes and the degree of polynomial spaces. The constant $C$ can take on different values in different occurrences.}

\begin{definition}
\label{def_discrete_H_m}
For any integers $m \geq 2$, we define the discrete $H^{m}$-norm $\Vert v \Vert_{m,h}$ by 
\begin{align*}
& \Vert v \Vert_{m, h}^{2} = \sum_{i=0}^{m} \Vert D^{i} v\Vert_{L^{2}(\calT_{h})}^{2} 
+ \sum_{j = 1}^{m-1} h^{-(2m - 2j - 1)} \Vert \lb D^{j}v \rb_{J}\Vert_{L^{2}(\calF_{h})}^{2} \\
:= & \sum_{i=0}^{m} \sum_{K \in \calT_{h}} \Vert D^{i} v\Vert_{L^{2}(K)}^{2} 
+ \sum_{j = 1}^{m-1} \sum_{F \in \calF_{h}} h^{-(2m - 2j - 1)}  \Vert \lb D^{j}v \rb_{J}\Vert_{L^{2}(F)}^{2}, 
\quad \forall v \in H_{0}^{1}(\Omega) \cap H^{m}(\calT_{h}).  
\end{align*}
For any $F \in \calF_{h}^{\text{int}}$, there are two elements $K^{-}, K^{+} \in \calT_{h}$ sharing 
the common face $F$. We denote by $ v^{-} := v|_{K^{-}}$ and 
$v^{+} := v|_{K^{+}}$.  
%$D^{j} v_{h}^{-} := D^{j}v_{h}|_{\partial K^{-}}$ and 
%$D^{j} v_{h}^{+} := D^{j}v_{h}|_{\partial K^{+}}$ for any $1 \leq j \leq m-1$. 
We define 
%\begin{align*}
%& \Vert \lb D^{j}v_{h} \rb_{J}\Vert_{L^{2}(F)} \\
%= & \left( \sum_{k_{1} = 1}^{d}\sum_{k_{2} = 1}^{d} \cdots \sum_{k_{j}=1}^{d} 
% \Vert \lb \partial_{x_{k_{1}}}\partial_{x_{k_{2}}}\cdots \partial_{x_{k_{d}}} v_{h} \rb_{J} \Vert^2 \right)^{\frac{1}{2}} \\
%%= & \Vert D^{j}v_{h}^{-} - D^{j} v_{h}^{+} \Vert_{L^{2}(F)} \\ 
%:= & \left( \sum_{k_{1} = 1}^{d}\sum_{k_{2} = 1}^{d} \cdots \sum_{k_{j}=1}^{d} 
%\Vert  (\partial_{x_{k_{1}}}\partial_{x_{k_{2}}}\cdots \partial_{x_{k_{d}}} v_{h}^{-})|_{F} 
%- (\partial_{x_{k_{1}}}\partial_{x_{k_{2}}}\cdots 
%\partial_{x_{k_{d}}} v_{h}^{+})|_{F} \Vert_{L^{2}(F)}^{2}\right)^{\frac{1}{2}}.
%\end{align*}
\begin{align*}
& \Vert \lb D^{j}v \rb_{J}\Vert_{L^{2}(F)}^{2} \\
= & \sum_{1\leq k_{1},\cdots, k_j \leq d}
 \Vert \lb \partial_{x_{k_{1}}}\partial_{x_{k_{2}}}\cdots \partial_{x_{k_{j}}} v \rb_{J} \Vert_{L^{2}(F)}^2 \\
%= & \Vert D^{j}v_{h}^{-} - D^{j} v_{h}^{+} \Vert_{L^{2}(F)} \\ 
:= & \sum_{1\leq k_{1},\cdots, k_j \leq d}
\Vert  (\partial_{x_{k_{1}}}\partial_{x_{k_{2}}}\cdots \partial_{x_{k_{j}}} v^{-})|_{F} 
- (\partial_{x_{k_{1}}}\partial_{x_{k_{2}}}\cdots 
\partial_{x_{k_{j}}} v^{+})|_{F} \Vert_{L^{2}(F)}^{2}.
\end{align*}
For any $F \in \calF_{h}^{\partial}$, we define 
%\begin{align*}
%& \Vert \lb D^{j}v_{h} \rb_{J}\Vert_{L^{2}(F)} \\
%= & \left( \sum_{k_{1} = 1}^{d}\sum_{k_{2} = 1}^{d} \cdots \sum_{k_{j}=1}^{d} 
% \Vert \lb \partial_{x_{k_{1}}}\partial_{x_{k_{2}}}\cdots \partial_{x_{k_{d}}} v_{h} \rb_{J} \Vert^2 \right)^{\frac{1}{2}} \\
%%= \Vert D^{j}v_{h} |_{F}\Vert_{L^{2}(F)} \\
%:= & \left( \sum_{k_{1} = 1}^{d}\sum_{k_{2} = 1}^{d} \cdots \sum_{k_{j}=1}^{d} 
%\Vert  (\partial_{x_{k_{1}}}\partial_{x_{k_{2}}}\cdots \partial_{x_{k_{d}}} v_{h})|_{F} 
%\Vert_{L^{2}(F)}^{2}\right)^{\frac{1}{2}}.
%\end{align*}
\begin{align*}
& \Vert \lb D^{j}v \rb_{J}\Vert_{L^{2}(F)}^{2} \\
= & \sum_{1\leq k_{1},\cdots, k_j \leq d}
 \Vert \lb \partial_{x_{k_{1}}}\partial_{x_{k_{2}}}\cdots \partial_{x_{k_{j}}} v \rb_{J} \Vert_{L^{2}(F)}^2  
%= \Vert D^{j}v_{h} |_{F}\Vert_{L^{2}(F)} \\
:= \sum_{1\leq k_{1},\cdots, k_j \leq d}
\Vert  (\partial_{x_{k_{1}}}\partial_{x_{k_{2}}}\cdots \partial_{x_{k_{j}}} v)|_{F} \Vert_{L^{2}(F)}^{2}.
\end{align*}
\end{definition}

\subsection{Discrete $H^{m}$-norm bounded by natural energy semi-norm}
\label{sec_thm_stab}
The main result of section~\ref{sec_thm_stab} is Theorem~\ref{thm_stab}, which shows that the discrete $H^{m}$-norm 
(see Definition~\ref{def_discrete_H_m}) is bounded by the natural energy semi-norm associated with the $C^{0}$ 
interior penalty method (\ref{IP_method}). The proof of Theorem~\ref{thm_stab} is based on Lemma~\ref{lemma_stab} 
and Lemma~\ref{lemma_D_m_interior}.

\begin{lemma}
\label{lemma_stab}
For any integers $r \geq m\geq 2$, there is a constant $C>0$ such that
\begin{align}
\label{stab_interface_bound}
\sum_{j=1}^{m-1} h^{-(2m - 2j - 1)}\Vert \lb D^{j}v_{h} \rb_{J} 
\Vert_{L^{2}(\calF_{h})}^{2} \leq C S_{h}(v_{h}, v_{h}), \qquad \forall v_{h} \in V_{h}.
\end{align}
\end{lemma}

\begin{proof}
We choose $F \in \calF_{h}$ arbitrarily. There is an orthonormal coordinate system $\{ y_{k}\}_{k=1}^{d}$ 
such that the $y_{d}$-axis is parallel to normal vector along $F$. Therefore $y_{1}$-axis, 
$\cdots$, $y_{d-1}$-axis are all parallel to $F$. 

We claim that for any $1 \leq l \leq m$, there is a positive integer $C^{\prime}$ such that 
\begin{align}
\label{stab_interface_ineq1}
\Vert \lb D^{l} \tilde{v}_{h} \rb_{J} \Vert_{L^{2}(F)}^{2} 
\leq C^{\prime}\left( \sum_{j=0}^{l} h^{-(2l -2j)} \Vert \lb \partial_{y_{d}}^{j} \tilde{v}_{h}  
\rb_{J}  \Vert_{L^{2}(F)}^{2} \right), \quad \forall \tilde{v}_{h} \in P_{r}(\calT_{h}).
\end{align}
{ We prove (\ref{stab_interface_ineq1}) by induction.} When $l=1$, it is easy to see 
\begin{align*}
\Vert \lb D \tilde{v}_{h} \rb_{J} \Vert_{L^{2}(F)}^{2} 
= \sum_{k=1}^{d} \Vert \lb \partial_{x_{k}} \tilde{v}_{h}\rb_{J} \Vert_{L^{2}(F)}^{2} 
= \sum_{k=1}^{d} \Vert \lb \partial_{y_{k}} \tilde{v}_{h}\rb_{J} \Vert_{L^{2}(F)}^{2}. 
\end{align*}
By discrete inverse inequality and the fact that $y_{1}$-axis, 
$\cdots$, $y_{d-1}$-axis are all parallel to $F$, we have that 
\begin{align*}
\sum_{k=1}^{d-1} \Vert \lb \partial_{y_{k}} \tilde{v}_{h}\rb_{J} \Vert_{L^{2}(F)}^{2}
\leq C h^{-2} \Vert \lb \tilde{v}_{h} \rb_{J} \Vert_{L^{2}(F)}^{2}.
\end{align*}
Therefore we have 
\begin{align*}
\Vert \lb D \tilde{v}_{h} \rb_{J} \Vert_{L^{2}(F)}^{2}  
\leq C \left( h^{-2} \Vert \lb \tilde{v}_{h} \rb_{J} \Vert_{L^{2}(F)}^{2} 
+ \Vert \lb \partial_{y_{d}} \tilde{v}_{h} \rb_{J} \Vert_{L^{2}(F)}^{2} \right).
\end{align*}
Thus (\ref{stab_interface_ineq1}) holds when $l=1$. 
We assume that (\ref{stab_interface_ineq1}) holds for any $l<m$. 
Then by discrete inverse inequality and the fact that $y_{1}$-axis, 
$\cdots$, $y_{d-1}$-axis are all parallel to $F$,
\begin{align*}
& \Vert \lb D^{l+1} \tilde{v}_{h} \rb_{J} \Vert_{L^{2}(F)}^{2} 
= \sum_{k=1}^{d} \Vert \lb \partial_{x_{k}} D^{l} \tilde{v}_{h} \rb_{J} \Vert_{L^{2}(F)}^{2}
= \sum_{k=1}^{d} \Vert \lb \partial_{y_{k}} D^{l} \tilde{v}_{h} \rb_{J} \Vert_{L^{2}(F)}^{2}\\
\leq & C \left(h^{-2} \Vert \lb D^{l} \tilde{v}_{h} \rb_{J} \Vert_{L^{2}(F)}^{2} 
+ \Vert \lb D^{l}(\partial_{y_{d}} \tilde{v}_{h}) \rb_J \Vert_{L^{2}(F)}^{2} \right).
\end{align*}
Since $\tilde{v}_{h} \in P_{r}(\calT_{h})$, then $\partial_{y_{d}} \tilde{v}_{h} \in P_{r}(\calT_{h})$. 
Since we assume (\ref{stab_interface_ineq1}) holds for $l$, we have 
\begin{align*}
& \Vert \lb D^{l}(\partial_{y_{d}} \tilde{v}_{h}) \rb_J \Vert_{L^{2}(F)}^{2} 
\leq C \left( \sum_{j=0}^{l} h^{-(2l -2j)} \Vert \lb \partial_{y_{d}}^{j} 
(\partial_{y_{d}}\tilde{v}_{h}) \rb_{J} \Vert_{L^{2}(F)}^{2} \right), \\
& h^{-2}\Vert \lb D^{l} \tilde{v}_{h} \rb_{J} \Vert_{L^{2}(F)}^{2} 
\leq C\left( \sum_{j=0}^{l} h^{-(2(l+1) -2j)} \Vert \lb \partial_{y_{d}}^{j} \tilde{v}_{h}  
\rb_{J}  \Vert_{L^{2}(F)}^{2} \right).
\end{align*}
Therefore (\ref{stab_interface_ineq1}) holds for $l+1$. Thus we can conclude that 
the claim (\ref{stab_interface_ineq1}) is true. 

{ Now we start to prove (\ref{stab_interface_bound}) by induction.} Since $\Vert \lb \partial_{y_{d}} v_{h}\rb_{J} \Vert_{L^{2}(F)} 
= \Vert \lb \nabla v_{h} \rb \Vert_{L^{2}(F)}$, (\ref{stab_interface_ineq1}) and the fact 
$v_{h} \in H_{0}^{1}(\Omega)$ imply
\begin{align}
\label{stab_m_1}
\Vert \lb D v_{h} \rb_{J} \Vert_{L^{2}(F)} = \Vert \lb \nabla v_{h} \rb \Vert_{L^{2}(F)}.
\end{align}
{Since $F\in \calF_{h}$ is chosen arbitrarily, (\ref{stab_m_1}) implies that 
(\ref{stab_interface_bound}) holds when $m=2$}. 

Applying (\ref{stab_interface_ineq1}) with $l = 2$, we have 
\begin{align}
\label{D2_ineq}
 & \Vert \lb D^{2} v_{h} \rb_{J} \Vert_{L^{2}(F)}^{2} \\ 
\nonumber
\leq & C \left( h^{-4} \Vert \lb v_{h}\rb_{J} \Vert_{L^{2}(F)}^{2} 
+ h^{-2}\Vert \lb \partial_{y_{d}} v_{h} \rb_{J} \Vert_{L^{2}(F)}^{2} 
+ \Vert \lb \partial_{y_{d}}^{2} v_{h} \rb_{J} \Vert_{L^{2}(F)}^{2} \right) \\ 
\nonumber
= & C \left( h^{-2}\Vert \lb \nabla v_{h} \rb \Vert_{L^{2}(F)}^{2} 
+  \Vert \lb \partial_{y_{d}}^{2} v_{h} \rb_{J} \Vert_{L^{2}(F)}^{2}\right).
\end{align}
The last equality in (\ref{D2_ineq}) holds since $v_{h} \in H_{0}^{1}(\Omega)$ and 
$\Vert \lb \partial_{y_{d}} v_{h}\rb_{J} \Vert_{L^{2}(F)} 
= \Vert \lb \nabla v_{h} \rb \Vert_{L^{2}(F)}$. 
We notice that 
\begin{align}
\label{laplace_orthonormal}
\Delta v_{h} = \left( \partial_{y_{1}}^{2} v_{h} + \cdots + \partial_{y_{d-1}}^{2} v_{h} \right) 
+ \partial_{y_{d}}^{2} v_{h}. 
\end{align}
Since $y_{1}$-axis, $\cdots$, $y_{d-1}$-axis are all parallel to $F$, discrete inverse inequality implies 
\begin{align*}
& \Vert \lb (\partial_{y_{1}}^{2} v_{h} + \cdots + \partial_{y_{d-1}}^{2} v_{h}) 
\rb \Vert_{L^{2}(F)}^{2} \\
= & \Vert \lb(\partial_{y_{1}}^{2} v_{h} + \cdots + \partial_{y_{d-1}}^{2} v_{h})  
\rb_{J} \Vert_{L^{2}(F)}^{2} \\
\leq & C h^{-2} \Vert \lb D v_{h}\rb_{J} \Vert_{L^{2}(F)}^{2} 
= C h^{-2} \Vert \lb \nabla v_{h}\rb \Vert_{L^{2}(F)}^{2}.
\end{align*}
By (\ref{laplace_orthonormal}) and the above inequality, we have 
\begin{align}
\label{laplace_normal}
\Vert \lb \partial_{y_{d}}^{2} v_{h}\rb_{J} \Vert_{L^{2}(F)}^{2} 
\leq C \left( h^{-2} \Vert \lb \nabla v_{h}\rb \Vert_{L^{2}(F)}^{2} 
 + \Vert \lb \Delta v_{h} \rb \Vert_{L^{2}(F)}^{2} \right).
\end{align} 
By (\ref{D2_ineq}, \ref{laplace_normal}), we have 
\begin{align}
\label{stab_m_2}
\Vert \lb D^{2} v_{h} \rb_{J} \Vert_{L^{2}(F)}^{2} 
\leq C \left( h^{-2}\Vert \lb \nabla v_{h} \rb \Vert_{L^{2}(F)}^{2}  
+ \Vert \lb \Delta v_{h} \rb \Vert_{L^{2}(F)}^{2} \right).
\end{align}
Since $F\in \calF_{h}$ is chosen arbitrarily, (\ref{stab_m_1}, \ref{stab_m_2}) imply that 
(\ref{stab_interface_bound}) holds when {$m=3$}. 

We assume that $1 \leq l < m$ is an odd number, $l = 2\tilde{l} + 1$ and 
\begin{align}
\label{induction_odd}
& \Vert \lb D^{l}v_{h} \rb_{J} \Vert_{L^{2}(F)}^{2} \\
\nonumber
\leq & C \big(\sum_{i=0}^{\tilde{l}} h^{-4i} \Vert \lb \nabla \Delta^{\tilde{l}-i} v_{h}\rb \Vert_{L^{2}(F)}^{2} 
+ \sum_{i=0}^{\tilde{l} - 1} h^{-(4i+2)} \Vert \lb \Delta^{\tilde{l}-i} v_{h}\rb \Vert_{L^{2}(F)}^{2}\big).
%\\
%\nonumber
%\leq & C \big( h^{-4\tilde{l}}\Vert \lb \nabla v_{h}\rb  \Vert_{L^{2}(F)}^{2} 
%+ h^{-2(2\tilde{l} - 1)} \Vert \lb \Delta v_{h} \rb \Vert_{L^{2}(F)}^{2} 
%+ h^{-2(2\tilde{l} - 2)} \Vert \lb \nabla \Delta v_{h} \rb \Vert_{L^{2}(F)}^{2} + \cdots \\
%\nonumber
%& \qquad + h^{-4} \Vert \lb \nabla \Delta^{\tilde{l} - 1} v_{h} \rb \Vert_{L^{2}(F)}^{2} 
%+ h^{-2} \Vert \lb \Delta^{\tilde{l}} v_{h} \rb \Vert_{L^{2}(F)}^{2}  
%+ \Vert \lb \nabla \Delta^{\tilde{l}} v_{h} \rb \Vert_{L^{2}(F)}^{2} \big).
\end{align}
Then by applying (\ref{induction_odd}) for each $\partial_{y_{k}}v_{h}$, we have 
\begin{align*}
& \Vert \lb D^{l+1}v_{h} \rb_{J} \Vert_{L^{2}(F)}^{2} 
= \sum_{k = 1}^{d} \Vert \lb D^{l}(\partial_{x_{k}} v_{h}) \rb_{J}\Vert_{L^{2}(F)}^{2} 
= \sum_{k = 1}^{d} \Vert \lb D^{l}(\partial_{y_{k}}v_{h}) \rb_{J}\Vert_{L^{2}(F)}^{2} \\
\leq & C \sum_{k = 1}^{d} \left (\sum_{i=0}^{\tilde{l}} h^{-4i} \Vert \lb \nabla \Delta^{\tilde{l}-i} 
(\partial_{y_{k}} v_{h}) \rb \Vert_{L^{2}(F)}^{2} 
+ \sum_{i=0}^{\tilde{l} - 1} h^{-(4i+2)} \Vert \lb \Delta^{\tilde{l}-i} 
(\partial_{y_{k}} v_{h})\rb \Vert_{L^{2}(F)}^{2}\right).
% \\
%= & C \sum_{k = 1}^{d} \big( h^{-4\tilde{l}}\Vert \lb \nabla 
%(\partial_{y_{k}} v_{h})\rb  \Vert_{L^{2}(F)}^{2} 
%+ h^{-2(2\tilde{l} - 1)} \Vert \lb \Delta (\partial_{y_{k}} v_{h}) \rb \Vert_{L^{2}(F)}^{2} \\
%& \qquad + h^{-2(2\tilde{l} - 2)} \Vert \lb \nabla \Delta (\partial_{y_{k}} v_{h})
% \rb \Vert_{L^{2}(F)}^{2} + \cdots 
% + h^{-4} \Vert \lb \nabla \Delta^{\tilde{l} - 1} (\partial_{y_{k}} v_{h}) \rb \Vert_{L^{2}(F)}^{2} \\
%& \qquad + h^{-2} \Vert \lb \Delta^{\tilde{l}} (\partial_{y_{k}} v_{h}) \rb \Vert_{L^{2}(F)}^{2}  
%+ \Vert \lb \nabla \Delta^{\tilde{l}} (\partial_{y_{k}} v_{h}) \rb \Vert_{L^{2}(F)}^{2} \big).
\end{align*}
Here $2\tilde{l} + 1 = l$. 
Since $y_{1}$-axis, $\cdots$, $y_{d-1}$-axis are all parallel to $F$, discrete inverse inequality implies 
\begin{align*}
& \Vert \lb D^{l+1}v_{h} \rb_{J} \Vert_{L^{2}(F)}^{2}  \\
\leq & C h^{-2} \left(\sum_{i=0}^{\tilde{l}} h^{-4i} \Vert \lb \nabla \Delta^{\tilde{l}-i} 
v_{h} \rb \Vert_{L^{2}(F)}^{2} 
+ \sum_{i=0}^{\tilde{l} - 1} h^{-(4i+2)} \Vert \lb \Delta^{\tilde{l}-i} 
v_{h} \rb \Vert_{L^{2}(F)}^{2}\right) \\
& \qquad + C \left(\sum_{i=0}^{\tilde{l}} h^{-4i} \Vert \lb \nabla \Delta^{\tilde{l}-i} 
(\partial_{y_{d}} v_{h}) \rb \Vert_{L^{2}(F)}^{2} 
+ \sum_{i=0}^{\tilde{l} - 1} h^{-(4i+2)} \Vert \lb \Delta^{\tilde{l}-i} 
(\partial_{y_{d}} v_{h})\rb \Vert_{L^{2}(F)}^{2}\right).
%%%= & C h^{-2} \big( h^{-4\tilde{l}}\Vert \lb \nabla v_{h} \rb  \Vert_{L^{2}(F)}^{2} 
%%%+ h^{-2(2\tilde{l} - 1)} \Vert \lb \Delta v_{h} \rb \Vert_{L^{2}(F)}^{2} \\
%%%& \qquad + h^{-2(2\tilde{l} - 2)} \Vert \lb \nabla \Delta v_{h} \rb \Vert_{L^{2}(F)}^{2} + \cdots 
%%% + h^{-4} \Vert \lb \nabla \Delta^{\tilde{l} - 1} v_{h} \rb \Vert_{L^{2}(F)}^{2} \\
%%%& \qquad + h^{-2} \Vert \lb \Delta^{\tilde{l}} v_{h} \rb \Vert_{L^{2}(F)}^{2}  
%%%+ \Vert \lb \nabla \Delta^{\tilde{l}} v_{h}  \rb \Vert_{L^{2}(F)}^{2} \big) \\
%%%& \qquad + C \big( h^{-4\tilde{l}}\Vert \lb \nabla 
%%%(\partial_{y_{d}} v_{h})\rb  \Vert_{L^{2}(F)}^{2} 
%%%+ h^{-2(2\tilde{l} - 1)} \Vert \lb \Delta (\partial_{y_{d}} v_{h}) \rb \Vert_{L^{2}(F)}^{2} \\
%%%& \qquad + h^{-2(2\tilde{l} - 2)} \Vert \lb \nabla \Delta (\partial_{y_{d}} v_{h})
%%% \rb \Vert_{L^{2}(F)}^{2} + \cdots 
%%% + h^{-4} \Vert \lb \nabla \Delta^{\tilde{l} - 1} (\partial_{y_{d}} v_{h}) \rb \Vert_{L^{2}(F)}^{2} \\
%%%& \qquad + h^{-2} \Vert \lb \Delta^{\tilde{l}} (\partial_{y_{d}} v_{h}) \rb \Vert_{L^{2}(F)}^{2}  
%%%+ \Vert \lb \nabla \Delta^{\tilde{l}} (\partial_{y_{d}} v_{h}) \rb \Vert_{L^{2}(F)}^{2} \big).
\end{align*}
Again by the fact that $y_{1}$-axis, $\cdots$, $y_{d-1}$-axis are all parallel to $F$, we have that 
for any $0 \leq i \leq \tilde{l}$, 
\begin{align*}
& \Vert \lb \nabla \Delta^{\tilde{l}-i} (\partial_{y_{d}} v_{h}) \rb \Vert_{L^{2}(F)}^{2} 
= \Vert \lb \nabla (\partial_{y_{d}} \Delta^{\tilde{l}-i} v_{h}) \rb_{J} \Vert_{L^{2}(F)}^{2} \\
= & \left( \Vert \lb \partial_{y_{1}}(\partial_{y_{d}} \Delta^{\tilde{l}-i} v_{h}) \rb_{J} \Vert_{L^{2}(F)}^{2} 
+ \cdots + \Vert \lb \partial_{y_{d-1}} (\partial_{y_{d}} \Delta^{\tilde{l}-i} v_{h}) \rb_{J} \Vert_{L^{2}(F)}^{2}\right) 
+ \Vert \lb \partial_{y_{d}}^{2} (\Delta^{\tilde{l}-i} v_{h}) \rb_{J} \Vert_{L^{2}(F)}^{2} \\
\leq & C \left( h^{-2}  \Vert \lb \partial_{y_{d}} \Delta^{\tilde{l}-i} v_{h} \rb_{J} \Vert_{L^{2}(F)}^{2} 
+ \Vert \lb  \partial_{y_{d}}^{2} (\Delta^{\tilde{l}-i} v_{h}) \rb_{J} \Vert_{L^{2}(F)}^{2} \right)\\
\leq & C \left( h^{-2}  \Vert \lb \nabla \Delta^{\tilde{l}-i} v_{h} \rb_{J} \Vert_{L^{2}(F)}^{2} 
+ \Vert \lb  \partial_{y_{d}}^{2} (\Delta^{\tilde{l}-i} v_{h}) \rb_{J} \Vert_{L^{2}(F)}^{2} \right) \\
\leq & C \left( h^{-2}  \Vert \lb \nabla \Delta^{\tilde{l}-i} v_{h} \rb_{J} \Vert_{L^{2}(F)}^{2} 
+ \Vert \lb  \Delta^{\tilde{l}-i + 1} v_{h} \rb_{J} \Vert_{L^{2}(F)}^{2} \right).
\end{align*}
We have applied (\ref{laplace_normal}) for $\Delta^{\tilde{l}-i} v_{h}$ to obtain last inequality. 
We also notice that for any $0 \leq i \leq \tilde{l}-1$,
\begin{align*}
\Vert \lb \Delta^{\tilde{l} - i} (\partial_{y_{d}} v_{h}) \rb \Vert_{L^{2}(F)}^{2} 
= \Vert \lb \partial_{y_{d}}(\Delta^{\tilde{l} - i}v_{h}) \rb_{J} \Vert_{L^{2}(F)}^{2} 
= \Vert \lb \nabla \Delta^{\tilde{l} - i}v_{h} \rb \Vert_{L^{2}(F)}^{2}.
\end{align*}
%%%\begin{align*}
%%%& \Vert \lb \Delta (\partial_{y_{d}} v_{h}) \rb \Vert_{L^{2}(F)}^{2} 
%%%= \Vert \lb \partial_{y_{d}} (\Delta v_{h}) \rb_{J} \Vert_{L^{2}(F)}^{2} 
%%%= \Vert \lb \nabla \Delta  v_{h} \rb \Vert_{L^{2}(F)}^{2}, \\
%%%& \cdots \\
%%%& \Vert \lb \Delta^{\tilde{l}} (\partial_{y_{d}} v_{h}) \rb \Vert_{L^{2}(F)}^{2}  
%%%= \Vert \lb  \partial_{y_{d}} ( \Delta^{\tilde{l}} v_{h}) \rb_{J} \Vert_{L^{2}(F)}^{2} 
%%%=  \Vert \lb  \nabla \Delta^{\tilde{l}} v_{h} \rb \Vert_{L^{2}(F)}^{2}.
%%%\end{align*}
Therefore we have that (\ref{induction_odd}) implies 
\begin{align}
\label{stab_m_l_plus_1_even}
 \Vert \lb D^{l+1}v_{h} \rb_{J} \Vert_{L^{2}(F)}^{2} 
\leq C \left( \sum_{i=0}^{\tilde{l}} h^{-4i}\Vert \lb \Delta^{\tilde{l}- i +1} v_{h} \rb\Vert_{L^{2}(F)}^{2}  
+ \sum_{i=0}^{\tilde{l}} h^{-(4i+2)} \Vert \lb \nabla \Delta^{\tilde{l}-i}v_{h} \rb\Vert_{L^{2}(F)}^{2} \right),
%%%= & C \big( h^{-2l} \Vert \lb \nabla v_{h} \rb \Vert_{L^{2}(F)}^{2} 
%%%+ h^{- 2(l-1)} \Vert \lb \Delta v_{h} \rb \Vert_{L^{2}(F)}^{2} 
%%%+ h^{- 2(l-2)}\Vert \lb \nabla \Delta v_{h} \rb \Vert_{L^{2}(F)}^{2} + \cdots \\
%%%& \qquad + h^{-2} \Vert \lb  \nabla \Delta^{\frac{l-1}{2}} v_{h} \rb \Vert_{L^{2}(F)}^{2} 
%%%+ \Vert \lb \Delta^{\frac{l+1}{2}} v_{h} \rb \Vert_{L^{2}(F)}^{2}\big),
\end{align}
where $l = 2\tilde{l}+1$. 

Now we assume that $1 \leq {l } < m$ is an even number, $l = 2\tilde{l}$ and 
\begin{align}
\label{induction_even}
\Vert \lb D^{l}v_{h} \rb_{J} \Vert_{L^{2}(F)}^{2} 
\leq C \left( \sum_{i=0}^{\tilde{l}-1} h^{-4i}\Vert \lb \Delta^{\tilde{l} - i}v_{h}  \rb \Vert_{L^{2}(F)}^{2} 
+ \sum_{i=0}^{\tilde{l} - 1} h^{-(4i+2)}\Vert \lb \nabla \Delta^{\tilde{l}-i-1} v_{h} \rb \Vert_{L^{2}(F)}^{2} \right). 
%%%= & C \big( h^{-2(2\tilde{l} - 1)}\Vert \lb \nabla v_{h}\rb  \Vert_{L^{2}(F)}^{2} 
%%%+ h^{-2(2\tilde{l} - 2)} \Vert \lb \Delta v_{h} \rb \Vert_{L^{2}(F)}^{2} 
%%%+ h^{-2(2\tilde{l} - 3)} \Vert \lb \nabla \Delta v_{h} \rb \Vert_{L^{2}(F)}^{2} + \cdots \\
%%%& \qquad + h^{-2}\Vert \lb \nabla \Delta^{\tilde{l} - 1} v_{h} \rb \Vert_{L^{2}(F)}^{2} 
%%%+ \Vert \lb \Delta^{\tilde{l}} v_{h} \rb \Vert_{L^{2}(F)}^{2}  \big).
\end{align}
Then by similar argument in last paragraph, we have that (\ref{induction_even}) implies 
\begin{align}
\label{stab_m_l_plus_1_odd}
\Vert \lb D^{l+1}v_{h} \rb_{J} \Vert_{L^{2}(F)}^{2}
\leq C \left(\sum_{i=0}^{\tilde{l}} h^{-4i}\Vert \lb \nabla \Delta^{\tilde{l} - i} v_{h}\rb \Vert_{L^{2}(F)}^{2}
+ \sum_{i=0}^{\tilde{l}-1} h^{-(4i+2)}\Vert \lb \Delta^{\tilde{l} - i} v_{h} \rb \Vert_{L^{2}(F)}^{2} \right),
\end{align}
where $l = 2 \tilde{l}$. 

According to (\ref{stab_m_1}, \ref{stab_m_2}, \ref{induction_odd}, \ref{stab_m_l_plus_1_even}, \ref{induction_even}, 
\ref{stab_m_l_plus_1_odd}) and the fact that $F \in \calF_{h}$ is chosen arbitrarily, we can conclude that 
the proof is complete.
\end{proof}

According to \cite[($3.1c$) in Theorem~$3.1$]{ChenPaniQiu}, there is a constant $C>0$ such that 
\begin{align}
\label{discrete_H2_stab}
& \Vert \nabla \tilde{v}_{h}\Vert_{L^{2}(\calT_{h})}^{2} 
+ \Vert D^{2} \tilde{v}_{h}\Vert_{L^{2}(\calT_{h})}^{2} \\
\nonumber 
\leq & C \left( \Vert \Delta \tilde{v}_{h}\Vert_{L^{2}(\calT_{h})}^{2} 
+ h^{-1} \Vert \lb \nabla \tilde{v}_{h} \rb \Vert_{L^{2}(\calF_{h})}^{2} 
+ h^{-3} \Vert \lb \tilde{v}_{h} \rb \Vert_{L^{2}(\calF_{h})}^{2} \right), 
\quad \forall \tilde{v}_{h} \in P_{\tilde{r}}(\calT_{h}), 
\end{align}
where $\tilde{r} \geq 2$ is a positive integer.

\begin{lemma}
\label{lemma_D_m_interior}
We define $2\tilde{m} + 1 = m$ if $m$ is an odd number, while $2 \tilde{m} = m$ if $m$ is an even number. 
Then there is a positive constant $C$ such that 
\begin{align}
\label{D_m_interior_bound}
\Vert D^{m}v_{h}\Vert_{L^{2}(\calT_{h})}^{2} \leq
\begin{cases}
& C \left( \Vert \Delta^{\tilde{m}} v_{h} \Vert_{L^{2}(\calT_{h})}^{2} 
 + \sum_{j=1}^{m-1} h^{-(2m - 2j - 1)}\Vert \lb D^{j}v_{h} \rb_{J} 
\Vert_{L^{2}(\calF_{h})}^{2}\right), \text{ if } m = 2\tilde{m}; \\
& \\
& C \left( \Vert \nabla \Delta^{\tilde{m}} v_{h} \Vert_{L^{2}(\calT_{h})}^{2} 
 + \sum_{j=1}^{m-1} h^{-(2m - 2j - 1)}\Vert \lb D^{j}v_{h} \rb_{J} 
\Vert_{L^{2}(\calF_{h})}^{2}\right), \text{ if } m = 2\tilde{m}+1,
\end{cases}
\end{align}
for any $v_{h} \in V_{h}$. 
\end{lemma}

\begin{proof}
It is easy to see that (\ref{D_m_interior_bound}) holds when $m=1$. 
By (\ref{discrete_H2_stab}), (\ref{D_m_interior_bound}) holds when $m=2$.

It is easy to see 
\begin{align*}
\Vert D^{3}v_{h}\Vert_{L^{2}(\calT_{h})}^{2} 
= \sum_{k=1}^{d} \Vert D^{2}(\partial_{x_{k}} v_{h})\Vert_{L^{2}(\calT_{h})}^{2}.
\end{align*}
Applying (\ref{discrete_H2_stab}) to each $\partial_{x_{k}} v_{h}$, we have 
%\begin{align*}
%& \Vert D^{3}v_{h}\Vert_{L^{2}(\calT_{h})}^{2} \\
%\leq & C \sum_{k=1}^{d}\big( \Vert \partial_{x_{k}} \Delta v_{h} \Vert_{L^{2}(\calT_{h})}^{2} 
%+ h^{-5}\Vert \lb \partial_{x_{k}} v_{h} \rb_{J} \Vert_{L^{2}(\calF_{h})}^{2} \\
%& \qquad + h^{-3}\Vert \lb D (\partial_{x_{k}} v_{h}) \rb_{J} \Vert_{L^{2}(\calF_{h})}^{2}
%+ h^{-1}\Vert \lb D^{2} (\partial_{x_{k}} v_{h}) \rb_{J} \Vert_{L^{2}(\calF_{h})}^{2}\big) \\
%= & C \big( \Vert \nabla \Delta v_{h} \Vert_{L^{2}(\calT_{h})}^{2} 
%+ h^{-5}\Vert \lb D v_{h} \rb_{J} \Vert_{L^{2}(\calF_{h})}^{2} \\ 
%& \qquad + h^{-3}\Vert \lb D^{2} v_{h} \rb_{J} \Vert_{L^{2}(\calF_{h})}^{2} 
%+ h^{-1}\Vert \lb D^{3} v_{h} \rb_{J} \Vert_{L^{2}(\calF_{h})}^{2} \big). 
%\end{align*}
\begin{align*}
& \Vert D^{3}v_{h}\Vert_{L^{2}(\calT_{h})}^{2} \\
\leq & C \sum_{k=1}^{d}\left( \Vert \Delta (\partial_{x_{k}} v_{h}) \Vert_{L^{2}(\calT_{h})}^{2} 
+ h^{-3}\Vert \lb  \partial_{x_{k}} v_{h} \rb_{J} \Vert_{L^{2}(\calF_{h})}^{2}
+ h^{-1}\Vert \lb D (\partial_{x_{k}} v_{h}) \rb_{J} \Vert_{L^{2}(\calF_{h})}^{2}\right) \\
= & C \left( \Vert \nabla \Delta v_{h} \Vert_{L^{2}(\calT_{h})}^{2} 
  + h^{-3}\Vert \lb D v_{h} \rb_{J} \Vert_{L^{2}(\calF_{h})}^{2} 
+ h^{-1}\Vert \lb D^{2} v_{h} \rb_{J} \Vert_{L^{2}(\calF_{h})}^{2} \right).
\end{align*}
Thus (\ref{D_m_interior_bound}) holds when $m=3$. 

For any $2 < l \leq m$, we have 
\begin{align*}
\Vert D^{l} v_{h}\Vert_{L^{2}(\calT_{h})}^{2} 
= \Vert D^{2} (D^{l-2} v_{h}) \Vert_{\calT_{h}}^{2}.
\end{align*}
Applying (\ref{discrete_H2_stab}) to each component of $D^{l-2} v_{h}$, we have 
%\begin{align}
%\label{D_m_interior_bound_recursive}
%& \Vert D^{l} v_{h}\Vert_{L^{2}(\calT_{h})}^{2} \\
%\nonumber
%\leq & C \big( \Vert D^{l-2} \Delta v_{h} \Vert_{L^{2}(\calT_{h})}^{2}  
%+ h^{-5}\Vert \lb D^{l-2} v_{h} \rb_{J} \Vert_{L^{2}(\calF_{h})}^{2} \\
%\nonumber 
%& \qquad + h^{-3}\Vert \lb D (D^{l-2} v_{h}) \rb_{J} \Vert_{L^{2}(\calF_{h})}^{2} 
%+ h^{-1}\Vert \lb D^{2} ( D^{l-2} v_{h}) \rb_{J} \Vert_{L^{2}(\calF_{h})}^{2} \big) \\
%\nonumber
%\leq & C \big( \Vert D^{l-2} \Delta v_{h} \Vert_{L^{2}(\calT_{h})}^{2}  
%+ h^{-5}\Vert \lb D^{l-2} v_{h} \rb_{J} \Vert_{L^{2}(\calF_{h})}^{2} \\
%\nonumber 
%& \qquad + h^{-3}\Vert \lb D^{l-1} v_{h} \rb_{J} \Vert_{L^{2}(\calF_{h})}^{2} 
%+ h^{-1}\Vert \lb D^{l} v_{h} \rb_{J} \Vert_{L^{2}(\calF_{h})}^{2} \big).
%\end{align}
\begin{align}
\label{D_m_interior_bound_recursive}
& \Vert D^{l} v_{h}\Vert_{L^{2}(\calT_{h})}^{2} \\
\nonumber
\leq & C \left( \Vert D^{l-2} \Delta v_{h} \Vert_{L^{2}(\calT_{h})}^{2}  
+ h^{-3}\Vert \lb D^{l-2} v_{h} \rb_{J} \Vert_{L^{2}(\calF_{h})}^{2} 
+ h^{-1}\Vert \lb D^{l-1} v_{h} \rb_{J} \Vert_{L^{2}(\calF_{h})}^{2} \right).
\end{align}

According to (\ref{D_m_interior_bound_recursive}) and the fact (\ref{D_m_interior_bound}) holds for $m=1,2,3$, we can conclude that the proof is complete.
\end{proof}

According to Lemma~\ref{lemma_stab}, Lemma~\ref{lemma_D_m_interior} 
and the discrete Poincar\'e inequality, we immediately have the following 
Theorem~\ref{thm_stab}.

\begin{theorem}
\label{thm_stab}
For any integers $r \geq m\geq 1$, there is a constant $C>0$ such that 
\begin{align*}
\Vert v_{h}\Vert_{m, h}^{2} \leq 
\begin{cases}
&C \left( \Vert \Delta^{\tilde{m}} v_{h}\Vert_{L^{2}(\calT_{h})}^{2} 
+ S_{h}(v_{h}, v_{h})\right),\quad \text{ if } m = 2 \tilde{m}; \\
& \\
&C \left( \Vert \nabla \Delta^{\tilde{m}} v_{h}\Vert_{L^{2}(\calT_{h})}^{2} 
+ S_{h}(v_{h}, v_{h})\right),\quad \text{ if } m = 2 \tilde{m} + 1,
\end{cases}
\end{align*}
for any $v_{h} \in V_{h}$. $\Vert v_{h}\Vert_{m, h}$ is introduced in Definition~\ref{def_discrete_H_m}. 
We point out that the right hand side of the above inequality is the natural energy semi-norm 
associated with the method (\ref{IP_method}). 
\end{theorem}

\subsection{Energy estimate of $C^{0}$ interior penalty method}
We provide Theorem~\ref{thm_wellposed}, which shows energy estimate of $C^{0}$ interior penalty method (\ref{IP_method}) 
with respect to the discrete $H^{m}$-norm (see Definition~\ref{def_discrete_H_m}). 
Before we prove Theorem~\ref{thm_wellposed}, we introduce Lemma~\ref{lemma_wellposed}. 

\begin{lemma}
\label{lemma_wellposed}
For any integers $r \geq m \geq 2$ and any spatial dimension $d \geq 1$, there is a positive number 
$\tau_{0} \geq 1$ such that for any $v_{h} \in V_{h}$,
\begin{align}
\label{wellposed_ineq}
& 4  |C_{h}(v_{h}, v_{h})| \\
\nonumber
\leq & 
\begin{cases}
&\Vert \Delta^{\tilde{m}} v_{h}\Vert_{L^{2}(\calT_{h})}^{2} + \tau_{0} S_{h}(v_{h}, v_{h}), 
\quad \text{ if } m = 2 \tilde{m} \quad (m \text{ is an even number}); \\
&\\
&\Vert \nabla \Delta^{\tilde{m}} v_{h}\Vert_{L^{2}(\calT_{h})}^{2} + \tau_{0} S_{h}(v_{h}, v_{h}), 
\quad \text{ if } m = 2 \tilde{m}+1 \quad (m \text{ is an odd number}).
\end{cases}
\end{align}
\end{lemma}

\begin{proof}
We prove (\ref{wellposed_ineq}) for $m = 2\tilde{m}$ ($m$ is an even number) in the following. 
It is similar to prove (\ref{wellposed_ineq}) for $m$ which is an odd number.

According to Definition~\ref{def_couple}, discrete trace inequality 
and inverse inequality,  
\begin{align*}
& |C_{h}(v_{h}, v_{h}) |\\ 
= & \left|- \sum_{i=0}^{\tilde{m} -1} \langle \Lb \Delta^{\tilde{m} + i} v_{h} \Rb, 
\lb \nabla \Delta^{\tilde{m} - i -1}v_{h} \rb \rangle_{\calF_{h}} 
+ \sum_{i=0}^{\tilde{m} - 2} \langle \Lb \nabla \Delta^{\tilde{m} + i} v_{h} \Rb, 
 \lb \Delta^{\tilde{m} - i-1} v_{h}\rb  \rangle_{\calF_{h}} \right| \\
\leq & \sum_{i=0}^{\tilde{m} - 1} \Vert \Lb \Delta^{\tilde{m} + i} v_{h} \Rb \Vert_{L^{2}(\calF_{h})} 
\Vert \lb \nabla \Delta^{\tilde{m} - i -1}v_{h} \rb \Vert_{L^{2}(\calF_{h})} \\
& \qquad + \sum_{i=0}^{\tilde{m} - 2} \Vert \Lb \nabla \Delta^{\tilde{m} + i} v_{h} \Rb \Vert_{L^{2}(\calF_{h})} 
\Vert \lb \Delta^{\tilde{m} - i-1} v_{h}\rb \Vert_{L^{2}(\calF_{h})} \\
\leq & C \sum_{i=0}^{\tilde{m} - 1} h^{-\frac{1}{2}}\Vert \Delta^{\tilde{m} + i} v_{h}\Vert_{L^{2}(\calT_{h})} 
\Vert \lb \nabla \Delta^{\tilde{m} - i -1}v_{h} \rb \Vert_{L^{2}(\calF_{h})} \\
& \qquad + C \sum_{i=0}^{\tilde{m} - 2} h^{-\frac{1}{2}} \Vert \nabla \Delta^{\tilde{m} + i} v_{h} \Vert_{L^{2}(\calT_{h})} 
\Vert \lb \Delta^{\tilde{m} - i-1} v_{h}\rb \Vert_{L^{2}(\calF_{h})} \\
\leq & C \sum_{i=0}^{\tilde{m} - 1} h^{-(2i + \frac{1}{2})}\Vert \Delta^{\tilde{m}} v_{h}\Vert_{L^{2}(\calT_{h})} 
\Vert \lb \nabla \Delta^{\tilde{m} - i -1}v_{h} \rb \Vert_{L^{2}(\calF_{h})} \\
& \qquad + C \sum_{i=0}^{\tilde{m} - 2} h^{-(2i+\frac{3}{2})} \Vert \Delta^{\tilde{m}} v_{h} \Vert_{L^{2}(\calT_{h})} 
\Vert \lb \Delta^{\tilde{m} - i-1} v_{h}\rb \Vert_{L^{2}(\calF_{h})} \\
\leq & \dfrac{1}{4}\Vert \Delta^{\tilde{m}} v_{h}\Vert_{L^{2}(\calT_{h})}^{2}  
+ C \left( \sum_{i=0}^{\tilde{m} - 1} h^{-(4i+ 1)} \Vert \lb \nabla \Delta^{\tilde{m} - i -1}v_{h} 
\rb \Vert_{L^{2}(\calF_{h})}^{2} + \sum_{i=0}^{\tilde{m} - 2} h^{-(4i+3)} 
\Vert \lb \Delta^{\tilde{m} - i-1} v_{h}\rb \Vert_{L^{2}(\calF_{h})}^{2} \right). 
\end{align*} 
By Definition~\ref{def_stab_new}, it is easy to see that (\ref{wellposed_ineq}) holds. Therefore the proof is complete. 
\end{proof}

\begin{theorem}
\label{thm_wellposed}
When $m=1$, the method (\ref{IP_method}) is well-posed such that 
\begin{align}
\label{enery_bound_m_1}
\Vert u_{h} \Vert_{H^{1}(\Omega)} \leq C \Vert f\Vert_{H^{-1}(\Omega)}.
\end{align} 
For any $m \geq 2$, there is a positive number $\tau_{0}\geq 1$ which is the same as Lemma~\ref{lemma_wellposed}, 
such that if $\tau \geq \tau_{0}$, then the method (\ref{IP_method}) is well-posed such that 
\begin{align}
\label{energy_bound_m_geq_2}
\Vert u_{h}\Vert_{m, h} \leq C \Vert f \Vert_{H^{-1}(\Omega)}.
\end{align}
Here $u_{h} \in V_{h}$ is the numerical solution of the method (\ref{IP_method}). 
\end{theorem}

\begin{proof}
By (\ref{method_m_1}), the method is the standard finite element method for Poisson equation when $m=1$. Therefore, 
the method (\ref{IP_method}) is well-posed and (\ref{enery_bound_m_1}) holds, when $m=1$. 

Now we consider $m \geq 2$. { By the definition of the bilinear form $a_h(\cdot,\cdot)$, Theorem \ref{thm_stab} and Lemma \ref{lemma_wellposed}, the coercivity and the continuity of $a_h(\cdot,\cdot)$ are obtained which imply the well-posedness of the method (\ref{IP_method}).} We assume $m = 2\tilde{m}$ to be an even number.
By taking $v_{h} = u_{h}$ in the method (\ref{IP_method}), we have 
\begin{align*}
\Vert \Delta^{\tilde{m}} u_{h} \Vert_{L^{2}(\calT_{h})}^{2} + 2 C_{h}(u_{h}, u_{h}) 
+ \tau S_{h}(u_{h}, u_{h}) = (f, u_{h})_{\Omega}.
\end{align*}
We choose $\tau_{0}$ the same as Lemma~\ref{lemma_wellposed}. Then Lemma~\ref{lemma_wellposed} implies 
\begin{align*}
\dfrac{1}{2}\Vert \Delta^{\tilde{m}} u_{h} \Vert_{L^{2}(\calT_{h})}^{2} 
+ \dfrac{\tau}{2} S_{h}(u_{h}, u_{h}) \leq (f, u_{h})_{\Omega} 
\leq \Vert f\Vert_{H^{-1}(\Omega)} \Vert u_{h}\Vert_{H^{1}(\Omega)} 
\leq \Vert f\Vert_{H^{-1}(\Omega)} \Vert u_{h}\Vert_{m,h},
\end{align*}
if $\tau \geq \tau_{0}$. Then by Theorem~\ref{thm_stab} and the above inequality, 
{we obtain (\ref{energy_bound_m_geq_2}) when $m$ is an even number.} 

It is similar to show that (\ref{energy_bound_m_geq_2}) holds when $m$ is an odd number. Thus we can conclude that 
the proof is complete.
\end{proof}

\subsection{Error analysis of $C^{0}$ interior penalty method}

We provide Theorem~\ref{thm_conv}, which gives error analysis of $C^{0}$ interior penalty method (\ref{IP_method}) 
with respect to the discrete $H^{m}$-norm (see Definition~\ref{def_discrete_H_m}).

\begin{theorem}
\label{thm_conv}
We assume that the exact solution $u \in H_{0}^{m}(\Omega) \cap H^{s}(\Omega)$ where $s \geq 2m - 1$.
When $m=1$, we have
\begin{align}
\label{final_error_m_1}
\Vert u - u_{h}\Vert_{H^{1}(\Omega)} \leq C h^{min(r, s-1)} \Vert u\Vert_{H^{s}(\Omega)}.
\end{align}
{For $m\geq 2$, we assume that $\tau \geq \tau_{0} \geq 1$} where $\tau_{0}$ 
is the same as Theorem~\ref{thm_wellposed}. Then we have 
\begin{align}
\label{final_error}
\Vert u - u_{h}\Vert_{m, h} \leq C h^{\min (r + 1 - m, s - m)} \Vert u \Vert_{H^{s}(\Omega)}. 
\end{align}
Here $u_{h} \in V_{h}$ is the numerical solution of the method (\ref{IP_method}). 
\end{theorem}

\begin{proof}
When $m=1$, the method (\ref{IP_method}) is the standard finite element method (\ref{method_m_1}) for 
Poisson equation with homogeneous Dirichlet boundary condition. So it is easy to see that (\ref{final_error_m_1}) holds.
In the following, we assume $m \geq 2$.

By Theorem~\ref{thm_wellposed}, the method (\ref{IP_method}) has the unique numerical solution $u_{h} \in V_{h}$. 

Since $u \in H_{0}^{m}(\Omega)$, it is easy to see that for any $0 \leq j \leq m-1$, 
every component of $D^{j}u$ is {continuous} across any face $F \in \calF_{h}^{\text{int}}$ 
and is equal to zero along $\partial\Omega$.
Therefore by Definition~\ref{def_couple} and Definition~\ref{def_stab_new}, we have 
\begin{align}
\label{err_consistent}
C_{h} (v_{h}, u) = S_{h}(u, v_{h}) = 0, \qquad \forall v_{h} \in V_{h}.
\end{align}

We denote by $\Pi_{h} u \in V_{h}$ the standard $L^{2}$-orthogonal projection of $u$ into $V_{h}$. 
We define $e_{u} = \Pi_{h} u - u_{h}$ and $\delta_{u} = \Pi_{h}u - u$. Since $u \in H^{s}(\Omega)$ 
and $s \geq 2m - 1$, we have 
\begin{align}
\label{approximation_props}
\Vert D^{j} \delta_{u} \Vert_{L^{2}(\calT_{h})} 
\leq C h^{\max ( \min (r + 1 - j, s -j), 0)} \Vert u\Vert_{H^{s}(\Omega)}, 
\quad \forall 0 \leq j \leq 2m -1.
\end{align}

We assume $m = 2\tilde{m}$ to be an even number. By (\ref{u_eq1}, \ref{err_consistent}) and 
the method (\ref{IP_method}), we have 
\begin{align}
\label{err_eq1}
& \Vert \Delta^{\tilde{m}} e_{u} \Vert_{L^{2}(\calT_{h})}^{2} + 2 C_{h}(e_{u}, e_{u}) + \tau S_{h}(e_{u}, e_{u}) \\
\nonumber
= & \left( \Delta^{\tilde{m}} \delta_{u}, \Delta^{\tilde{m}} e_{u} \right)_{\calT_{h}}
+ C_{h}(\delta_{u}, e_{u}) + C_{h} (e_{u}, \delta_{u}) + \tau S_{h}(\delta_{u}, e_{u}).
\end{align} 
By Lemma~\ref{lemma_wellposed} and (\ref{err_eq1}), we have 
\begin{align}
\label{err_eq2}
& \dfrac{1}{2} \Vert \Delta^{\tilde{m}} e_{u} \Vert_{L^{2}(\calT_{h})}^{2} + \dfrac{\tau}{2} S_{h}(e_{u}, e_{u}) \\
\nonumber
\leq & \left( \Delta^{\tilde{m}} \delta_{u}, \Delta^{\tilde{m}} e_{u} \right)_{\calT_{h}}
+ C_{h}(\delta_{u}, e_{u}) + C_{h} (e_{u}, \delta_{u}) + \tau S_{h}(\delta_{u}, e_{u}).
\end{align}

It is easy to see that 
\begin{align*}
& \left( \Delta^{\tilde{m}} \delta_{u}, \Delta^{\tilde{m}} e_{u} \right)_{\calT_{h}} \\
\leq  &\Vert \Delta^{\tilde{m}} \delta_{u} \Vert_{L^{2}(\calT_{h})} 
\Vert \Delta^{\tilde{m}} e_{u} \Vert_{\calT_{h}}
\leq C h^{\min (r + 1 - m, s - m)} \Vert \Delta^{\tilde{m}} e_{u} \Vert_{\calT_{h}} \Vert u\Vert_{H^{s}(\Omega)}, \\
& \tau S_{h}(\delta_{u}, e_{u}) \\
= & \tau \sum_{i=0}^{\tilde{m} -1} h^{-(4i + 1)} \langle \lb \nabla 
\Delta^{\tilde{m} - i -1}\delta_{u} \rb, \lb \nabla \Delta^{\tilde{m} - i -1} e_{u} \rb \rangle_{\calF_{h}} \\
& \qquad + \tau \sum_{i=0}^{\tilde{m} - 2} h^{-(4i + 3)} \langle \lb \Delta^{\tilde{m} - i-1}\delta_{u} \rb, 
 \lb \Delta^{\tilde{m} - i-1} e_{u}\rb  \rangle_{\calF_{h}} \\
\leq & \tau \sum_{i=0}^{\tilde{m} -1} \left( h^{-(4i + 1)} \Vert \lb \nabla 
\Delta^{\tilde{m} - i -1}\delta_{u} \rb \Vert_{L^{2}(\calF_{h})}^{2}  \right)^{\frac{1}{2}} 
\left( h^{-(4i + 1)} \Vert \lb \nabla \Delta^{\tilde{m} - i -1} e_{u} \rb 
\Vert_{L^{2}(\calF_{h})}^{2}  \right)^{\frac{1}{2}} \\
& \qquad + \tau \sum_{i=0}^{\tilde{m} - 2} \left( h^{-(4i + 3)} \Vert \lb 
\Delta^{\tilde{m} - i -1}\delta_{u} \rb\Vert_{L^{2}(\calF_{h})}^{2} \right)^{\frac{1}{2}} 
\left( h^{-(4i + 3)} \Vert \lb \Delta^{\tilde{m} - i -1} e_{u} 
\rb\Vert_{L^{2}(\calF_{h})}^{2} \right)^{\frac{1}{2}} \\
\leq & C \tau h^{\min (r + 1 - m, s - m)} \left( S_{h}(e_{u}, e_{u}) \right)^{\frac{1}{2}} 
\Vert u \Vert_{H^{s}(\Omega)}. 
\end{align*}
We have used trace inequality and (\ref{approximation_props}) to obtain the last inequality above.

By trace inequality  and (\ref{approximation_props}) again, we have 
\begin{align*}
& C_{h}(\delta_{u}, e_{u}) \\
= & - \sum_{i=0}^{\tilde{m} -1} \langle \Lb \Delta^{\tilde{m} + i} \delta_{u} \Rb, 
\lb \nabla \Delta^{\tilde{m} - i -1} e_{u} \rb \rangle_{\calF_{h}} 
+ \sum_{i=0}^{\tilde{m} - 2} \langle \Lb \nabla \Delta^{\tilde{m} + i} \delta_{u} \Rb, 
 \lb \Delta^{\tilde{m} - i-1} e_{u} \rb  \rangle_{\calF_{h}} \\
\leq & \sum_{i=0}^{\tilde{m} -1} \left( h^{4i + 1} \Vert \Lb \Delta^{\tilde{m} + i} \delta_{u} \Rb 
\Vert_{L^{2}(\calF_{h})}^{2} \right)^{\frac{1}{2}} \left( h^{-(4i + 1)} 
\Vert \lb \nabla \Delta^{\tilde{m} - i -1} e_{u} \rb \Vert_{L^{2}(\calF_{h})}^{2} \right)^{\frac{1}{2}} \\
 & \qquad + \sum_{i=0}^{\tilde{m} - 2} \left( h^{4i+3} \Vert \Lb \nabla\Delta^{\tilde{m} + i} \delta_{u} \Rb 
\Vert_{L^{2}(\calF_{h})}^{2} \right)^{\frac{1}{2}} \left( h^{-(4i+3)} \Vert \lb \Delta^{\tilde{m} - i-1} e_{u} \rb 
\Vert_{L^{2}(\calF_{h})}^{2} \right)^{\frac{1}{2}} \\
\leq & C\sum_{i=0}^{\tilde{m} -1} \left( h^{4i + 1} (h^{-1}\Vert \Delta^{\tilde{m} + i} \delta_{u}  
\Vert_{L^{2}(\calT_{h})}^{2} + h \Vert \Delta^{\tilde{m} + i} \delta_{u}  
\Vert_{H^{1}(\calT_{h})}^{2}) \right)^{\frac{1}{2}} \left( h^{-(4i + 1)} 
\Vert \lb \nabla \Delta^{\tilde{m} - i -1} e_{u} \rb \Vert_{L^{2}(\calF_{h})}^{2} \right)^{\frac{1}{2}} \\
 & \quad + C \sum_{i=0}^{\tilde{m} - 2} \left( h^{4i+3} ( h^{-1}\Vert \nabla\Delta^{\tilde{m} + i} \delta_{u} 
\Vert_{L^{2}(\calT_{h})}^{2} + h \Vert \nabla\Delta^{\tilde{m} + i} \delta_{u} 
\Vert_{H^{1}(\calT_{h})}^{2}) \right)^{\frac{1}{2}} \left( h^{-(4i+3)} \Vert \lb \Delta^{\tilde{m} - i-1} e_{u} \rb 
\Vert_{L^{2}(\calF_{h})}^{2} \right)^{\frac{1}{2}} \\
\leq & C h^{\min (r + 1 - m, s - m)} \left( S_{h}(e_{u}, e_{u}) \right)^{\frac{1}{2}} \Vert u\Vert_{H^{s}(\Omega)}.
\end{align*}
By inverse inequality and discrete trace inequality, 
\begin{align*}
& C_{h}(e_{u}, \delta_{u}) \\
= & - \sum_{i=0}^{\tilde{m} -1} \langle \Lb \Delta^{\tilde{m} + i} e_{u} \Rb, 
\lb \nabla \Delta^{\tilde{m} - i -1} \delta_{u} \rb \rangle_{\calF_{h}} 
+ \sum_{i=0}^{\tilde{m} - 2} \langle \Lb \nabla \Delta^{\tilde{m} + i} e_{u} \Rb, 
 \lb \Delta^{\tilde{m} - i-1} \delta_{u} \rb  \rangle_{\calF_{h}} \\
\leq & \sum_{i=0}^{\tilde{m} -1} \left( h^{4i + 1} \Vert \Lb \Delta^{\tilde{m} + i} e_{u} \Rb 
\Vert_{L^{2}(\calF_{h})}^{2} \right)^{\frac{1}{2}} \left( h^{-(4i + 1)} 
\Vert \lb \nabla \Delta^{\tilde{m} - i -1} \delta_{u} \rb \Vert_{L^{2}(\calF_{h})}^{2} \right)^{\frac{1}{2}} \\
 & \qquad + \sum_{i=0}^{\tilde{m} - 2} \left( h^{4i+3} \Vert \Lb \nabla\Delta^{\tilde{m} + i} e_{u} \Rb 
\Vert_{L^{2}(\calF_{h})}^{2} \right)^{\frac{1}{2}} \left( h^{-(4i+3)} \Vert \lb 
\Delta^{\tilde{m} - i-1} \delta_{u} \rb \Vert_{L^{2}(\calF_{h})}^{2} \right)^{\frac{1}{2}} \\
\leq & C\sum_{i=0}^{\tilde{m} -1} \left( h^{4i} \Vert \Delta^{\tilde{m} + i} e_{u} 
\Vert_{L^{2}(\calT_{h})}^{2} \right)^{\frac{1}{2}} \left( h^{-(4i + 1)} 
\Vert \lb \nabla \Delta^{\tilde{m} - i -1} \delta_{u} \rb \Vert_{L^{2}(\calF_{h})}^{2} \right)^{\frac{1}{2}} \\
 & \qquad + C \sum_{i=0}^{\tilde{m} - 2} \left( h^{4i+2} \Vert \nabla\Delta^{\tilde{m} + i} e_{u} 
\Vert_{L^{2}(\calT_{h})}^{2} \right)^{\frac{1}{2}} \left( h^{-(4i+3)} \Vert \lb 
\Delta^{\tilde{m} - i-1} \delta_{u} \rb \Vert_{L^{2}(\calF_{h})}^{2} \right)^{\frac{1}{2}} \\
\leq & C\sum_{i=0}^{\tilde{m} -1} \left( \Vert \Delta^{\tilde{m}} e_{u} 
\Vert_{L^{2}(\calT_{h})}^{2} \right)^{\frac{1}{2}} \left( h^{-(4i + 1)} 
\Vert \lb \nabla \Delta^{\tilde{m} - i -1} \delta_{u} \rb \Vert_{L^{2}(\calF_{h})}^{2} \right)^{\frac{1}{2}} \\
 & \qquad + C \sum_{i=0}^{\tilde{m} - 2} \left( \Vert \Delta^{\tilde{m}} e_{u} 
\Vert_{L^{2}(\calT_{h})}^{2} \right)^{\frac{1}{2}} \left( h^{-(4i+3)} \Vert \lb 
\Delta^{\tilde{m} - i-1} \delta_{u} \rb \Vert_{L^{2}(\calF_{h})}^{2} \right)^{\frac{1}{2}} \\
\leq & C h^{\min (r + 1 - m, s - m)} \Vert \Delta^{\tilde{m}} e_{u} \Vert_{L^{2}(\calT_{h})}^{2}
\Vert u \Vert_{H^{s}(\Omega)}. 
\end{align*}

Combing (\ref{err_eq2}) with above estimates, we have 
\begin{align*}
\Vert \Delta^{\tilde{m}} e_{u} \Vert_{L^{2}(\calT_{h})}^{2} +  \tau S_{h}(e_{u}, e_{u})
\leq C h^{2 \min (r + 1 - m, s - m)} \Vert u \Vert_{H^{s}(\Omega)}^{2}.  
\end{align*}
We have used the fact that $\tau$ is independent of $h$ to obtain the above inequality.
Then by {Theorem~\ref{thm_stab}}, we have 
\begin{align*}
\Vert e_{u}\Vert_{m, h}^{2} \leq C 
\left( \Vert \Delta^{\tilde{m}} e_{u} \Vert_{L^{2}(\calT_{h})}^{2} +  \tau S_{h}(e_{u}, e_{u}) \right) 
\leq C h^{2 \min (r + 1 - m, s - m)} \Vert u \Vert_{H^{s}(\Omega)}^{2}.
\end{align*}
Now we obtain the error estimate (\ref{final_error}) when $m\geq 2$ is an even number. It is similar to show that (\ref{final_error}) holds when $m\geq 2$ is an odd number. Therefore we can conclude that the proof is complete.
\end{proof}

{ 
\section{Error analysis under the low regularity assumption}
In the above section, we assume the exact solution $u \in H_{0}^{m}(\Omega) \cap H^{s}(\Omega)$ with $s \geq 2m - 1$. Since this regularity assumption may be high for the realistic problems, we further deduce the error analysis in this section for the exact solution under the low regularity assumption, i.e., $u \in H_{0}^{m}(\Omega)$. In this case, the Galerkin orthogonality does not hold true for the $C^0$ interior penalty method if $m\geq 2$. We derive the error analysis by the technique developed by Gudi in \cite{Gudi2010} which utilizes the analysis idea from the {\em a posteriori} error analysis.

Let $V:=H_{0}^{m}(\Omega)$ and $V^c_h $ be the $H^m$-conforming finite element space in $V$. One can refer to the construction 
of $H^m$-conforming finite element space in $V$ in any dimension according to a recent work in \cite{HuLinWu2021}.  For any 
$v,w\in V$, let $a(v,w) =(\Delta^{\tilde{m}} v, \Delta^{\tilde{m}} w)_{\Omega} $ if $m=2\tilde{m}$ and 
$a(v,w) = (\nabla \Delta^{\tilde{m}} v, \nabla \Delta^{\tilde{m}} w)_{\Omega}$ if $m=2\tilde{m}+1$.  
As the three abstract assumptions in \cite{Gudi2010}, firstly we assume there exists an enriching operator 
$E_h: V_h \rightarrow V^c_h$ such that
\begin{align}
\label{Ass_1}
\sum_{K\in \calT_{h}} h^{-2m}_K \|v - E_h v\|^2_{L^2(K)} + \|E_h v\|^2_V \leq C \|v\|^2_{m,h}, \quad \forall v\in V_h.
\end{align}
Actually, for the cases of $m=2$ and $3$, this enriching operator $E_h$ has been constructed by averaging technique and the above estimate has been derived in \cite{Gudi2010} and \cite{GudiNeilan2011}.

Secondly, by the definition of $a_h(\cdot,\cdot)$ in (\ref{IP_method_bilinear_form}) and Lemma \ref{lemma_wellposed}, choosing $\tau$ as in Theorem \ref{thm_wellposed}, we easily have that
\begin{align}
\label{Ass_2}
\|v_h\|^2_{m,h} \leq C a_h(v_h,v_h),\quad \forall v_h \in V_h.
\end{align}

Thirdly, we have the following estimate: for any $v\in V$, $w\in V^c_h$ and $v_h \in V_h$, it holds that
\begin{align}
\label{Ass_3}
|a(v,w) - a_h(v_h,w)| \leq C \|v-v_h\|_{m,h}\|w\|_V.
\end{align}
Actually, for $m=2\tilde{m}$, due to the fact that $w\in V^c_h$ and $v\in V$, we can derive 
\begin{align}
a(v,w) - a_h(v_h,w) &= (\Delta^{\tilde{m}} v,\Delta^{\tilde{m}} w) -  (\Delta^{\tilde{m}} v_h,\Delta^{\tilde{m}} w) -C_h(v_h,w)- C_h(w,v_h) - \tau S_h(v_h,w) \nn\\
& =  (\Delta^{\tilde{m}} (v-v_h),\Delta^{\tilde{m}} w) - C_h(w,v_h)\nn\\
& = (\Delta^{\tilde{m}} (v-v_h),\Delta^{\tilde{m}} w) - C_h(w,v-v_h). \label{er_a_a_h}
\end{align}
By the trace inequality and inverse estimate, we have
\begin{align*}
 &\quad C_h(w,v-v_h) \\
&=  -\sum_{i=0}^{\tilde{m} -1} \langle \Lb \Delta^{\tilde{m} + i} w \Rb, 
\lb \nabla \Delta^{\tilde{m} - i -1} (v-v_{h}) \rb \rangle_{\calF_{h}} 
+ \sum_{i=0}^{\tilde{m} - 2} \langle \Lb \nabla \Delta^{\tilde{m} + i} w \Rb, 
 \lb \Delta^{\tilde{m} - i-1} (v-v_{h})\rb  \rangle_{\calF_{h}}\\
 & \leq C \sum_{i=0}^{\tilde{m} -1}  \| \Delta^{\tilde{m}} w \|_{\calT_{h}}  h^{-\frac{1}{2}-2i}
 \| \lb \nabla \Delta^{\tilde{m} - i -1} (v-v_{h}) \rb \|_{\calF_{h}}\\
& \quad  + C \sum_{i=0}^{\tilde{m} - 2}  \| \Delta^{\tilde{m}} w \|_{\calT_{h}}  h^{-\frac{3}{2}-2i} \| \lb   \Delta^{\tilde{m} - i -1} (v-v_{h}) \rb \|_{\calF_{h}},
\end{align*} 
which, together with (\ref{er_a_a_h}), yields the estimate (\ref{Ass_3}). For the case of $m=2\tilde{m}+1$, one can similarly derive (\ref{Ass_3}) and we omit the details here.

By the estimates (\ref{Ass_1},\ref{Ass_2},\ref{Ass_3}) and following Lemma 2.1 in \cite{Gudi2010}, we have
\begin{align}\label{key_est}
\|u-u_h\|_{m,h} \leq C \inf_{v\in V_h}\left( \|u-v\|_{m,h}  +  \sup_{\phi \in V_h\setminus \{0\}} \frac{(f,\phi-E_h \phi)_\Omega - a_h(v,\phi-E_h \phi)}{\|\phi\|_{m,h}}\right).
\end{align}
In order to get the upper bound for the second term on the right-hand side of (\ref{key_est}), we first provide two lemmas.

\begin{lemma}\label{est_residual}
Let $v\in V_h$. There exists a positive constant $C$ independent of mesh size such that
\begin{align*}
\sum_{K \in \calT_{h}} h^{2m} \| f-(-1)^m\Delta^mv \|^2_{L^2(K)} \leq \begin{cases}  C\left( \|\Delta^{\tilde{m}}(u-v)\|^2_{\calT_{h}} + {\rm osc}^2_m(f)   \right),\quad \ \ {\rm if} \ m=2\tilde{m}, \\  C\left( \|\nabla\Delta^{\tilde{m}}(u-v)\|^2_{\calT_{h}} + {\rm osc}^2_m(f)   \right),\quad {\rm if} \ m=2\tilde{m}+1,
\end{cases}
\end{align*}
where 
\begin{align}\label{def_osc_f}
{\rm osc}_m(f) = \left( \sum_{K \in\calT_{h} } h^{2m}_K \inf_{\overline{f} \in P_{r-m}(K)}{ \|f - \overline{f}\|^2_{L^2(K)} }  \right)^{\frac{1}{2}}.
\end{align}
\end{lemma}
\begin{proof}
We provide the proof for the case of $m=2\tilde{m}$, and the case of $m=2\tilde{m}+1$ can be similarly deduced. Let $b_K \in  P_{m(d+1)}(K)\cap H^m_0(K)$ be the bubble function defined on $K$ such that $b_K (x_K ) = 1$, where $x_K$ is the barycenter of the element $K$. Let $\psi = b_K ( \overline{f}-(-1)^m\Delta^mv  )$ on $K \in \calT_{h}$ and $\psi = 0$ on $\Omega \setminus K$. We easily have that 
\begin{align}\label{est_bubble_el}
C_1 \|\overline{f}-(-1)^m\Delta^mv \|_{L^2(K)} \leq \|\psi\|_{L^2(K)} \leq C_2 \|\overline{f}-(-1)^m\Delta^mv \|_{L^2(K)}.
\end{align}
It follows integration by parts that
\begin{align*}
(f-(-1)^m\Delta^mv,\psi)_K = ((-1)^m\Delta^m u - (-1)^m\Delta^mv,\psi)_K = (\Delta^{\tilde{m}}(u-v),\Delta^{\tilde{m}}\psi)_K.
\end{align*}
By the inverse estimate, we further have
\begin{align*}
C \|\overline{f}-(-1)^m\Delta^mv \|^2_{L^2(K)} &\leq (\overline{f}-(-1)^m\Delta^mv,\psi)_K \\
& = (\overline{f} -f,\psi )_K +  (f-(-1)^m\Delta^mv,\psi)_K  \\
&  = (\overline{f} -f,\psi )_K +(\Delta^{\tilde{m}}(u-v),\Delta^{\tilde{m}}\psi)_K\\
& \leq C \left( \|f- \overline{f}\|_{L^2(K)} + h^{-m}_K \|  \Delta^{\tilde{m}}(u-v)\|_{L^2(K)}  \right) \|\psi\|_{L^2(K)},
\end{align*}
which, together with (\ref{est_bubble_el}), yields that
\[
h^{m}_K \|\overline{f}-(-1)^m\Delta^mv \|_{L^2(K)} \leq C  \left(h^{m}_K  \|f- \overline{f}\|_{L^2(K)} + \|  \Delta^{\tilde{m}}(u-v)\|_{L^2(K)}  \right).
\]
By the above estimate and the triangular inequality, we directly obtain
\[
h^{2m} \| f-(-1)^m\Delta^mv \|^2_{L^2(K)} \leq C \left(h^{2m}_K  \|f- \overline{f}\|^2_{L^2(K)} + \|  \Delta^{\tilde{m}}(u-v)\|^2_{L^2(K)}  \right),
\]
which yields the desired estimate.
\end{proof}

\begin{lemma}\label{lemma_jump_est}
Let $v\in V_h$. For $m=2\tilde{m}$, there exists a positive constant $C$ independent of mesh size such that, for $i=0,\cdots, \tilde{m}-1$,
\begin{align}
\label{jum_even_1}
\sum_{F \in \calF_h} h^{4i+1}_F \| \lb \Delta^{\tilde{m}+i} v \rb \|^2_{L^2(F)} &\leq C \left( \|\Delta^{\tilde{m}}(u-v)\|^2_{\calT_{h}} + {\rm osc}^2_m(f)   \right),\\
\label{jum_even_2}
\sum_{F \in \calF_h} h^{4i+3}_F \| \lb \nabla \Delta^{\tilde{m}+i} v \rb \|^2_{L^2(F)} &\leq C \left( \|\Delta^{\tilde{m}}(u-v)\|^2_{\calT_{h}} + {\rm osc}^2_m(f)   \right).
\end{align}
For $m=2\tilde{m}+1$, there exists a positive constant $C$ independent of mesh size such that
\begin{align}
\label{jum_odd_1}
\sum_{F \in \calF_h} h^{4i+3}_F \| \lb \Delta^{\tilde{m}+i+1} v \rb \|^2_{L^2(F)} &\leq C \left( \|\nabla\Delta^{\tilde{m}}(u-v)\|^2_{\calT_{h}} + {\rm osc}^2_m(f)   \right), \quad i=0,\cdots, \tilde{m}-1, \\
\label{jum_odd_2}
\sum_{F \in \calF_h} h^{4i+1}_F \| \lb \nabla \Delta^{\tilde{m}+i} v \rb \|^2_{L^2(F)} &\leq C \left( \| \nabla\Delta^{\tilde{m}}(u-v)\|^2_{\calT_{h}} + {\rm osc}^2_m(f)   \right), \quad i=0,\cdots, \tilde{m}.
\end{align}

\end{lemma}
\begin{proof}
For brevity, we only provide the proof for the case of $m=2\tilde{m}$. The estimates (\ref{jum_odd_1}) and (\ref{jum_odd_2}) for the case of $m=2\tilde{m}+1$ can be similarly deduced. The proof is based on the induction approach. 

Now we prove (\ref{jum_even_1}) with $i=0$. For any $F \in \calF^{\rm int}_h$, we denote $\omega_F = K^- \cup K^+ $ where $\partial K^- \cap \partial K^+ = F$. Let $\bfnu_{F}$ be the unit normal vector along $F$ pointing from $K^-$ to $K^+$. Let $\xi_1 \in P_{r-1}(\omega_F)$ be defined by 
\begin{align}
\Delta^{\tilde{m}-i} \xi_1|_F = 0,\quad i=1,\cdots,\tilde{m}, \label{xi_1_1} \\
\nabla \Delta^{\tilde{m}-i} \xi_1 \cdot \bfnu_{F} |_F =0,\quad i=2,\cdots,\tilde{m}, \label{xi_1_2}\\
\nabla \Delta^{\tilde{m}-1} \xi_1 \cdot \bfnu_{F} |_F = \lb \Delta^{\tilde{m}}v \rb_J |_F.  \label{xi_1_3}
\end{align}
For the construction of $\xi_1$, we can firstly assume $r=m$.  Let $\lambda^+_{d+1}$ and $\lambda^-_{d+1}$ be the linear basis functions at the nodes opposite to the face $F$ on $K^+$ and $K^-$ respectively. We choose $\xi_1|_{K^+} = C^+ (\lambda^+_{d+1})^{m-1}$ and $\xi_1|_{K^-} = C^- (\lambda^-_{d+1})^{m-1}$, where $C^+$ and $C^-$ are constants. It is obviously that $\xi_1$ satisfies (\ref{xi_1_1}) and (\ref{xi_1_2}). One can easily choose $C^+$ and $C^-$ such that $$C^+ \nabla \Delta^{\tilde{m}-1} (\lambda^+_{d+1})^{m-1} \cdot  \bfnu_{F} |_{\partial K^+\cap F} = C^- \nabla \Delta^{\tilde{m}-1} (\lambda^-_{d+1})^{m-1} \cdot  \bfnu_{F} |_{\partial K^-\cap F}=\lb \Delta^{\tilde{m}}v \rb_J |_F.$$
For $r> m$, one can similarly construct $\xi_1|_{K^+} = (\lambda^+_{d+1})^{m-1} \xi_1^+$ and $\xi_1|_{K^-} = (\lambda^-_{d+1})^{m-1} \xi_1^-$, where $\xi_1^+\in P_{r-m}(K^+), \xi_1^-\in P_{r-m}(K^-)$ such that $\xi_1$ satisfies (\ref{xi_1_1}-\ref{xi_1_3}).

Let $\xi_2 \in H^m_0(\omega_F)$ be a piecewise polynomial bubble function such that $\xi_2(m_F) =1$, where $m_F$ is the barycenter of $F$. Denote $\phi = \xi_1\xi_2$ on $\omega_F$ and extend it by zero on $\Omega \setminus \omega_F$. It follows from the definitions of $\xi_1$, $\xi_2$ and integration by parts that
\begin{align*}
C \| \lb  \Delta^{\tilde{m}} v \rb \|^2_{L^2(F)} &\leq  \langle \lb  \Delta^{\tilde{m}} v \rb_J,  \xi_2  \nabla \Delta^{\tilde{m}-1} \xi_1 \cdot \bfnu_{F}  \rangle_F \\
& = \langle \lb  \Delta^{\tilde{m}} v \rb_J,     \nabla \Delta^{\tilde{m}-1} \phi \cdot \bfnu_{F}  \rangle_F\\
& = ( \Delta^{\tilde{m}}v, \Delta^{\tilde{m}}\phi )_{\omega_F} +  (\nabla \Delta^{\tilde{m}}v,\nabla \Delta^{\tilde{m}-1}\phi )_{\omega_F} \\
& =  ( \Delta^{\tilde{m}}v, \Delta^{\tilde{m}}\phi )_{\omega_F}   - ( \Delta^{2 \tilde{m} }v,\phi )_{\omega_F} \\
& =  ( \Delta^{\tilde{m}}(v-u), \Delta^{\tilde{m}}\phi )_{\omega_F}   + (f- \Delta^{m}v,\phi )_{\omega_F}.
\end{align*}
By the scaling argument, for $K \in \omega_F$, we have
\[
\|\xi_1\|_{L^\infty(K)} \leq C h_F^{m-\frac{1}{2}-\frac{d}{2}} \|\lb \Delta^{\tilde{m}}v \rb_J\|_{L^2(F)} = C h_F^{m-\frac{1}{2}-\frac{d}{2}} \|\lb \Delta^{\tilde{m}}v \rb\|_{L^2(F)},
\]
which directly yields that
\begin{align}\label{phi_es}
\|\phi\|_{L^2(K)} \leq C \|\xi_1\|_{L^\infty(K)} \|\xi_2\|_{L^2(K)}\leq  Ch_F^{m-\frac{1}{2}}  \|\lb \Delta^{\tilde{m}}v \rb\|_{L^2(F)}.
\end{align}
By the inverse estimate, we have
\begin{align}
\label{es_1}
 \| \lb  \Delta^{\tilde{m}} v \rb \|^2_{L^2(F)} &\leq C \sum_{K \in \omega_F} \left( h^{-m}_K \| \Delta^{\tilde{m}}(u-v) \|_{L^2(K)} + \| f-\Delta^mv \|_{L^2(K)} \right) \|\phi\|_{L^2(K)}.
\end{align}
Combining (\ref{phi_es}) and (\ref{es_1}) yields
\begin{align}\label{es_2}
h_F  \| \lb  \Delta^{\tilde{m}} v \rb \|^2_{L^2(F)} &\leq C \sum_{K \in \omega_F} \left(  \| \Delta^{\tilde{m}}(u-v) \|^2_{L^2(K)} + h^{2m}_F\| f-\Delta^mv \|_{L^2(K)} \right).
\end{align}
The above estimate (\ref{es_2}) can be similarly deduced for the case of $F\in \partial \Omega$. Now combining (\ref{es_2}), Lemma \ref{est_residual} and summing over all $F\in \calF_h$, we get the estimate  (\ref{jum_even_1}) with $i=0$.

Next we prove (\ref{jum_even_2}) with $i=0$. For any $F \in \calF_h^{\rm int}$, let $\eta_1 \in P_{r-3}(\omega_F)$ be defined by $\Delta^{\tilde{m}-i}\eta_1|_F = 0$, $\nabla \Delta^{\tilde{m}-i}\eta_1 \cdot \bfnu_{F} |_F = 0$ with $i=2,\cdots,\tilde{m}$, and $\Delta^{\tilde{m}-1}\eta_1|_F = \lb\nabla \Delta^{\tilde{m}}v   \rb|_F$.  Here $\eta_1$ can be similarly constructed as $\xi_1$.  Let $\eta_2 \in H^m_0(\omega_F)$ be a piecewise polynomial bubble function such that $\eta_2(m_F) =1$, where $m_F$ is the barycenter of $F$. Denote $\psi = \eta_1\eta_2$ on $\omega_F$ and extend it by zero on $\Omega \setminus \omega_F$. By integration by parts, we have
\begin{align*}
C \| \lb  \nabla\Delta^{\tilde{m}} v \rb \|^2_{L^2(F)} &\leq \langle \lb \nabla \Delta^{\tilde{m}} v \rb,  \eta_2   \Delta^{\tilde{m}-1} \eta_1   \rangle_F \\
& = \langle \lb \nabla \Delta^{\tilde{m}} v \rb,     \Delta^{\tilde{m}-1} \psi  \rangle_F\\
& = ( \nabla \Delta^{\tilde{m}}v, \nabla\Delta^{\tilde{m}-1}\psi )_{\omega_F} +  ( \Delta^{\tilde{m}+1}v,  \Delta^{\tilde{m}-1}\psi )_{\omega_F} \\
& = -(\Delta^{\tilde{m}}v, \Delta^{\tilde{m}}\psi )_{\omega_F}  + \sum_{K\in \omega_F} \int_{\partial K} \lb \Delta^{\tilde{m}}v \rb \Lb  \nabla 
\Delta^{\tilde{m}-1}\psi \Rb + ( \Delta^{2\tilde{m}} v,\psi)_{\omega_F} \\
& =  (\Delta^{\tilde{m}}(u-v), \Delta^{\tilde{m}}\psi )_{\omega_F}  + \sum_{K\in \omega_F} \int_{\partial K} \lb \Delta^{\tilde{m}}v \rb \Lb  \nabla 
\Delta^{\tilde{m}-1}\psi \Rb + ( \Delta^{2\tilde{m}} v-f,\psi)_{\omega_F} .
\end{align*}
By the scaling argument, for $K \in \omega_F$, we have
\[
\|\eta_1\|_{L^\infty(K)} \leq C h_F^{m-\frac{3}{2}-\frac{d}{2}} \|\lb \nabla \Delta^{\tilde{m}}v \rb\|_{L^2(F)},
\]
which yields
\begin{align}
\label{est_2_1}
\|\psi\|_{L^2(K)} \leq  Ch_F^{m-\frac{3}{2}}  \|\lb \nabla \Delta^{\tilde{m}}v \rb\|_{L^2(F)}.
\end{align}
By the inverse estimate and the trace inequality, we get
\begin{align*}
 &\quad h^3_F \| \lb  \nabla\Delta^{\tilde{m}} v \rb \|^2_{L^2(F)} \\
&\leq C \sum_{K\in \omega_F} \left(  \| \Delta^{\tilde{m}}(u-v) \|_{L^2(K)} + h^{m}_F\| f-\Delta^mv \|_{L^2(K)} + h^{\frac{1}{2}}_F \| \lb  \Delta^{\tilde{m}} v \rb \|_{L^2(\partial K)}  \right) h^{3-m}_F \|\psi\|_{L^2(K)} .
\end{align*}
The above estimate can be similarly deduced for the case of $F\in \partial \Omega$. Combining the above estimate with (\ref{est_2_1}), (\ref{es_2}), Lemma \ref{est_residual} and summing over all $F\in \calF_h$, we obtain the estimate (\ref{jum_even_2}) with $i=0$.

We assume (\ref{jum_even_1}) and (\ref{jum_even_2}) hold true for $0<i=k\leq \tilde{m}-2$, we would like to prove that (\ref{jum_even_1}) and (\ref{jum_even_2}) hold true with $i=k+1$. Since the derivations are similar, we only show the proof for (\ref{jum_even_1}) with $i=k+1$. 

For any $F\in \calF_h^{\rm int}$, let $\gamma_1 \in P_{r-4k-5}(\omega_F)$ be defined by $ \Delta^{\tilde{m}-l} \gamma_1|_F = 0$ with $l=k+2,\cdots,\tilde{m}$, $\nabla \Delta^{\tilde{m}-l} \gamma_1 \cdot \bfnu_{F} |_F =0 $ with $l=k+3,\cdots,\tilde{m}$, $\nabla \Delta^{\tilde{m}-k-2} \gamma_1 \cdot \bfnu_{F} |_F = \lb \Delta^{\tilde{m}+k+1}v \rb_J |_F$.  Here $\gamma_1$ can be similarly constructed as $\xi_1$.  Let $\gamma_2 \in H^m_0(\omega_F)$ be a piecewise polynomial bubble function such that 
$\gamma_2(m_F) =1$, where $m_F$ is the barycenter of $F$. Denote $\varphi = \gamma_1\gamma_2$ on $\omega_F$ 
and extend it by zero on $\Omega \setminus \omega_F$. By integration by parts, we have
\begin{align}
C \| \lb \Delta^{\tilde{m}+k+1} v \rb \|^2_{L^2(F)} &\leq  \langle \lb  \Delta^{\tilde{m}+k+1} v \rb_J,  \gamma_2  \nabla \Delta^{\tilde{m}-k-2} \gamma_1 \cdot \bfnu_{F}  \rangle_F \nn\\ 
& = \langle \lb  \Delta^{\tilde{m}+k+1} v \rb_J,    \nabla \Delta^{\tilde{m}-k-2} \varphi \cdot \bfnu_{F}  \rangle_F \nn \\
& = \left(   \Delta^{\tilde{m}+k+1}v, \Delta^{\tilde{m}-k-1}\varphi \right)_{\omega_F} +  \left(\nabla \Delta^{\tilde{m}+k+1}v,\nabla \Delta^{\tilde{m}-k-2}\varphi \right)_{\omega_F}. \label{4_7_ind_1}
\end{align}
By the definition of $\varphi$, integration by parts yields
\begin{align}
(\nabla \Delta^{\tilde{m}+k+1}v,\nabla \Delta^{\tilde{m}-k-2}\varphi )_{\omega_F} = -( \Delta^{2\tilde{m}} v,\varphi )_{\omega_F}.  \label{4_7_ind_2}
\end{align}
We also have
\begin{align}
\left(   \Delta^{\tilde{m}+k+1}v, \Delta^{\tilde{m}-k-1}\varphi \right)_{\omega_F} &= \left(   \Delta^{\tilde{m}+k}v, \Delta^{\tilde{m}-k}\varphi \right)_{\omega_F} -\int_F \lb \Delta^{\tilde{m}+k}v  \rb_J \nabla \Delta^{\tilde{m}-k-1}\varphi \cdot \bfnu_F \nn  \\
&\quad + \int_F \lb \nabla \Delta^{\tilde{m}+k}v  \rb \Delta^{\tilde{m}-k-1}\varphi \nn \\
&:=T_0+T_1+T_2.\label{4_7_ind_3}
\end{align} 
By the scaling argument, for $K \in \omega_F$, we have
\[
\|\gamma_1\|_{L^\infty(K)} \leq C h_F^{2\tilde{m}-2k -\frac{5}{2}-\frac{d}{2}} \| \lb  \Delta^{\tilde{m}+k+1} v \rb\|_{L^2(F)},
\]
which yields
\begin{align}
\label{induction_1}
\| \varphi \|_{L^2(K)}\leq   C h_F^{2\tilde{m}-2k -\frac{5}{2}} \| \lb  \Delta^{\tilde{m}+k+1} v \rb\|_{L^2(F)}.
\end{align}
By the trace inequality, inverse estimate and (\ref{induction_1}), we obtain
\begin{align*}
T_1 & \leq  C \| \lb \Delta^{\tilde{m}+k}v  \rb\|_{L^2(F)} h_F^{-\frac{1}{2}} \|\nabla \Delta^{\tilde{m}-k-1}\varphi\|_{L^2(\omega_F)}\\
&\leq C \| \lb \Delta^{\tilde{m}+k}v  \rb\|_{L^2(F)}h_F^{-\frac{1}{2}} h_F^{-(2\tilde{m}-2k-1)} \|\varphi\|_{L^2(\omega_F)}\\
& \leq C h^{-2}_F \| \lb \Delta^{\tilde{m}+k}v  \rb\|_{L^2(F)} \|\lb  \Delta^{\tilde{m}+k+1} v \rb\|_{L^2(F)}.
\end{align*}
%Thus, we further have
%\begin{align*}
%   T_1 \leq \delta \| \lb \Delta^{\tilde{m}+k+1}v  \rb\|^2_{L^2(F)}  + \frac{C}{\delta} h^{-4}_F \|\lb  \Delta^{\tilde{m}+k} v \rb\|^2_{L^2(F)},
%\end{align*}
%where $\delta$ is an optional parameter to be determined. 
Similarly, by the trace inequality, inverse estimate and (\ref{induction_1}), we have
\begin{align*}
  T_2 \leq Ch^{-1}_F \|\lb \nabla \Delta^{\tilde{m}+k} v \rb\|_{L^2(F)} \| \lb \Delta^{\tilde{m}+k+1}v  \rb\|_{L^2(F)}   .
\end{align*}
%\begin{align*}
%  T_2 \leq \delta  \| \lb \Delta^{\tilde{m}+k+1}v  \rb\|^2_{L^2(F)}  + \frac{C}{\delta} h^{-2}_F \|\lb \nabla \Delta^{\tilde{m}+k} v \rb\|^2_{L^2(F)}.
%\end{align*}
For the estimate of $T_0$, following integration by parts, the trace inequality, inverse estimate and (\ref{induction_1}) yields
\begin{align*}
T_0 & = (\Delta^{\tilde{m}}v,\Delta^{\tilde{m}}\varphi)_{\omega_F} -\sum_{l=0}^{k-1} \int_F \lb  \Delta^{\tilde{m}+l} v \rb_J \nabla \Delta^{\tilde{m}-l-1} \varphi \cdot\bfnu_F + \sum_{l=0}^{k-1} \int_F \lb \nabla \Delta^{\tilde{m}+l} v \rb \Delta^{\tilde{m}-l-1} \varphi,\\
&\leq (\Delta^{\tilde{m}}v,\Delta^{\tilde{m}}\varphi)_{\omega_F} + C \sum_{l=0}^{k-1}  \|  \lb  \Delta^{\tilde{m}+l} v \rb \|_{L^2(F)}h^{-\frac{1}{2}}_F  \|\nabla \Delta^{\tilde{m}-l-1} \varphi \|_{L^2(\omega_F)}     \\
&\quad + C \sum_{l=0}^{k-1}   \|  \lb \nabla \Delta^{\tilde{m}+l} v \rb \|_{L^2(F)}h^{-\frac{1}{2}}_F  \|  \Delta^{\tilde{m}-l-1} \varphi \|_{L^2(\omega_F)} \\
&\leq  (\Delta^{\tilde{m}}v,\Delta^{\tilde{m}}\varphi)_{\omega_F} \\
&\quad + C \sum_{l=0}^{k-1} \left(  h^{2l+\frac{1}{2}}_F  \|  \lb \nabla \Delta^{\tilde{m}+l} v \rb \|_{L^2(F)} + h^{2l+\frac{3}{2}}_F  \|  \lb \nabla \Delta^{\tilde{m}+l} v \rb \|_{L^2(F)}  \right) h^{-2k-\frac{5}{2}}_F  \| \lb \Delta^{\tilde{m}+k+1}v  \rb\|_{L^2(F)} .
\end{align*}
Combining (\ref{4_7_ind_1}), (\ref{4_7_ind_2}), (\ref{4_7_ind_3}) and the upper bounds for $T_0$, $T_1$ and $T_2$, we have
\begin{align}
\label{4_7_ind_4}
&C \| \lb \Delta^{\tilde{m}+k+1} v \rb \|^2_{L^2(F)} \leq -( \Delta^{2\tilde{m}} v,\varphi )_{\omega_F} + (\Delta^{\tilde{m}}v,\Delta^{\tilde{m}}\varphi)_{\omega_F}\\
&\quad + C \sum_{l=0}^{k-1} \left(  h^{2l+\frac{1}{2}}_F  \|  \lb \nabla \Delta^{\tilde{m}+l} v \rb \|_{L^2(F)} + h^{2l+\frac{3}{2}}_F  \|  \lb \nabla \Delta^{\tilde{m}+l} v \rb \|_{L^2(F)}  \right) h^{-2k-\frac{5}{2}}_F  \| \lb \Delta^{\tilde{m}+k+1}v  \rb\|_{L^2(F)} \nn\\
&\quad+C h^{-2}_F \| \lb \Delta^{\tilde{m}+k}v  \rb\|_{L^2(F)} \|\lb  \Delta^{\tilde{m}+k+1} v \rb\|_{L^2(F)} +Ch^{-1}_F   \|\lb \nabla \Delta^{\tilde{m}+k} v \rb\|_{L^2(F)}\| \lb \Delta^{\tilde{m}+k+1}v  \rb\|_{L^2(F)}  .\nn
\end{align}
We note that
\[
-( \Delta^{2\tilde{m}} v,\varphi )_{\omega_F} + (\Delta^{\tilde{m}}v,\Delta^{\tilde{m}}\varphi)_{\omega_F}=
(f- \Delta^{2\tilde{m}} v,\varphi )_{\omega_F} + (\Delta^{\tilde{m}}v - \Delta^{\tilde{m}}u ,\Delta^{\tilde{m}}\varphi)_{\omega_F}.
\]
For any $F\in \partial \Omega$, we can derive the similar estimate as (\ref{4_7_ind_4}). Now combining (\ref{4_7_ind_4}), the Cauchy-Schwarz inequality, the estimates of (\ref{jum_even_1}) and (\ref{jum_even_2}) with $i\leq k$, the inverse estimate and (\ref{induction_1}) and summing over all $F\in \calF_h$, we can finally obtain (\ref{jum_even_1}) with $i=k+1$.
\end{proof}

Now we can start to derive the upper bound for the second term on the right-hand side of (\ref{key_est}). We provide the estimate for the case of $m=2\tilde{m}$, and the estimate for the case of $m=2\tilde{m}+1$ can be similarly obtained. For any $\phi \in V_h \setminus \{0\}$, let $\zeta = \phi - E_h \phi$. By the definition of $a_h(\cdot,\cdot)$ in (\ref{IP_method_bilinear_form}) and integration by parts, we have
\begin{align*}
&\quad (f,\zeta)_\Omega - a_h(v,\zeta) \\
&= (f,\zeta)_\Omega -  (\Delta^{\tilde{m}} v, \Delta^{\tilde{m}} \zeta)_{\calT_{h}} - C_{h}( v,\zeta) - C_{h}(\zeta,v) 
-\tau S_{h}(v, \zeta)  \\
& = (f - (-1)^m \Delta^m v, \zeta)_\Omega - \sum^{\tilde{m}-1}_{i=0} \langle  \Lb \Delta^{\tilde{m}+i}v  \Rb , \lb  \nabla \Delta^{\tilde{m}-i-1} \zeta \rb \rangle_{\calF_h} - \sum^{\tilde{m}-1}_{i=0} \langle  \lb \Delta^{\tilde{m}+i}v  \rb , \Lb  \nabla \Delta^{\tilde{m}-i-1} \zeta \Rb \rangle_{\calF_h}\\
&\quad + \sum^{\tilde{m}-2}_{i=0} \langle  \Lb \nabla \Delta^{\tilde{m}+i}v  \Rb , \lb    \Delta^{\tilde{m}-i-1} \zeta \rb \rangle_{\calF_h} +
 \sum^{\tilde{m}-1}_{i=0} \langle  \lb \nabla \Delta^{\tilde{m}+i}v  \rb , \Lb    \Delta^{\tilde{m}-i-1} \zeta \Rb \rangle_{\calF_h}
 \\
 &\quad - C_{h}( v,\zeta) - C_{h}(\zeta,v) 
-\tau S_{h}(v, \zeta)  .
\end{align*}
By the definition of $C_h(\cdot,\cdot)$, the trace inequality, inverse estimate and (\ref{Ass_1}), we further have
\begin{align*}
&\quad (f,\zeta)_\Omega - a_h(v,\zeta) \\
& = (f - (-1)^m \Delta^m v, \zeta)_\Omega  - \sum^{\tilde{m}-1}_{i=0} \langle  \lb \Delta^{\tilde{m}+i}v  \rb , \Lb  \nabla \Delta^{\tilde{m}-i-1} \zeta \Rb \rangle_{\calF_h}  +
 \sum^{\tilde{m}-1}_{i=0} \langle  \lb \nabla \Delta^{\tilde{m}+i}v  \rb , \Lb    \Delta^{\tilde{m}-i-1} \zeta \Rb \rangle_{\calF_h}
 \\
&\quad+ \sum_{i=0}^{\tilde{m} -1} \langle \Lb \Delta^{\tilde{m} + i} \zeta \Rb, 
\lb \nabla \Delta^{\tilde{m} - i -1}v \rb \rangle_{\calF_{h}} 
- \sum_{i=0}^{\tilde{m} - 2} \langle \Lb \nabla \Delta^{\tilde{m} + i}\zeta\Rb, 
 \lb \Delta^{\tilde{m} - i-1}v \rb  \rangle_{\calF_{h}}\\
 &\quad  -\tau \sum_{i=0}^{\tilde{m} -1} h^{-(4i + 1)} \langle \lb \nabla \Delta^{\tilde{m} - i -1}v \rb, 
\lb \nabla \Delta^{\tilde{m} - i -1}\zeta \rb \rangle_{\calF_{h}} 
-\tau \sum_{i=0}^{\tilde{m} - 2} h^{-(4i + 3)} \langle \lb \Delta^{\tilde{m} - i-1} v\rb, 
 \lb \Delta^{\tilde{m} - i-1}\zeta \rb  \rangle_{\calF_{h}}\\
&\leq  \|f - (-1)^m \Delta^m v\|_{L^2(\calT_h)} \|\zeta\|_{L^2(\calT_h)} + C \sum^{\tilde{m}-1}_{i=0} \|\lb \Delta^{\tilde{m}+i}v  \rb\|_{L^2(\calF_h)}h^{-2\tilde{m}+2i+\frac{1}{2}}\|\zeta\|_{L^2(\calT_h)}\\
&\quad +  C\sum^{\tilde{m}-1}_{i=0}\|\lb \nabla \Delta^{\tilde{m}+i}v  \rb\|_{L^2(\calF_h)}h^{-2\tilde{m}+2i+\frac{3}{2}}\|\zeta\|_{L^2(\calT_h)}\\
&\quad+ C\sum^{\tilde{m}-1}_{i=0} \|\lb \nabla \Delta^{\tilde{m}-i-1}v  \rb\|_{L^2(\calF_h)} h^{-2\tilde{m}-2i-\frac{1}{2}}\|\zeta\|_{L^2(\calT_h)}\\
&\quad+ C\sum^{\tilde{m}-2}_{i=0} \|\lb   \Delta^{\tilde{m}-i-1}v  \rb\|_{L^2(\calF_h)} h^{-2\tilde{m}-2i-\frac{3}{2}}\|\zeta\|_{L^2(\calT_h)}\\
&\quad + C  \sum^{\tilde{m}-1}_{i=0}h^{-(4i + 1)} \|\lb \nabla \Delta^{\tilde{m} - i -1}v \rb\|_{L^2(\calF_h)} \|\lb\nabla \Delta^{\tilde{m} - i -1}\zeta \rb\|_{L^2(\calF_h)}\\
&\quad + C  \sum^{\tilde{m}-2}_{i=0}h^{-(4i + 3)} \|\lb   \Delta^{\tilde{m} - i -1}v \rb\|_{L^2(\calF_h)} \| \lb \Delta^{\tilde{m} - i -1}\zeta \rb\|_{L^2(\calF_h)}\\
& \leq Ch^m \|f - (-1)^m \Delta^m v\|_{L^2(\calT_h)} \|\phi\|_{m,h} + C \sum^{\tilde{m}-1}_{i=0}  h^{2i+\frac{1}{2}}\|\lb \Delta^{\tilde{m}+i}v  \rb\|_{L^2(\calF_h)}\|\phi\|_{m,h}\\
&\quad +  C\sum^{\tilde{m}-1}_{i=0}h^{2i+\frac{3}{2}}\|\lb \nabla \Delta^{\tilde{m}+i}v  \rb\|_{L^2(\calF_h)}\|\phi\|_{m,h}+C\sum^{\tilde{m}-1}_{i=0}h^{-2i-\frac{1}{2}} \|\lb \nabla \Delta^{\tilde{m}-i-1}v  \rb\|_{L^2(\calF_h)} \|\phi\|_{m,h}\\
&\quad+ C\sum^{\tilde{m}-2}_{i=0} h^{-2i-\frac{3}{2}}\|\lb   \Delta^{\tilde{m}-i-1}v  \rb\|_{L^2(\calF_h)} \|\phi\|_{m,h}\\
&\quad + C  \sum^{\tilde{m}-1}_{i=0}h^{-(4i + 1)} \|\lb \nabla \Delta^{\tilde{m} - i -1}v \rb\|_{L^2(\calF_h)} \|\lb\nabla \Delta^{\tilde{m} - i -1}\phi \rb\|_{L^2(\calF_h)}\\
&\quad + C  \sum^{\tilde{m}-2}_{i=0}h^{-(4i + 3)} \|\lb   \Delta^{\tilde{m} - i -1}v \rb\|_{L^2(\calF_h)} \| \lb \Delta^{\tilde{m} - i -1}\phi \rb\|_{L^2(\calF_h)}.
\end{align*}
Combining the above estimate, Lemma \ref{est_residual} and Lemma \ref{lemma_jump_est}, we directly have
\begin{align*}
(f,\zeta)_\Omega - a_h(v,\zeta) \leq C \left( \|\Delta^{\tilde{m}}(u-v)\|_{\calT_{h}} + {\rm osc}_m(f)  \right)  \|\phi\|_{m,h},
\end{align*}
which yields
\begin{align}
\label{key_est_2}
 \sup_{\phi \in V_h\setminus \{0\}} \frac{(f,\phi-E_h \phi)_\Omega - a_h(v,\phi-E_h \phi)}{\|\phi\|_{m,h}}  \leq  C \left( \| u-v\|_{m,h} + {\rm osc}_m(f)  \right) .
\end{align}
The above estimate (\ref{key_est_2}) can be similarly deduced for the case of $m=2\tilde{m}+1$ and we omit the details here.

Now by the estimates (\ref{key_est}) and (\ref{key_est_2}), we obtain the following convergence result for the $C^0$ interior penalty method for the $m$th-Laplace equation (\ref{m_Laplace_eqs}) with $m\geq 2$.
\begin{theorem}
For $m\geq 2$, we assume that the exact solution $u \in H_{0}^{m}(\Omega)$ for (\ref{m_Laplace_eqs}), $\tau \geq \tau_{0} \geq 1$ 
where $\tau_{0}$ is the same as Theorem~\ref{thm_wellposed},  and there exists an enriching operator $E_h: V_h \rightarrow V^c_h$ 
such that (\ref{Ass_1}) holds true. Then there is a positive constant $C$ independent of $h$ such that
\begin{align}
\label{final_error_Hm}
\Vert u - u_{h}\Vert_{m, h} \leq C   \inf_{v\in V_h}  \left( \| u-v\|_{m,h} + {\rm osc}_m(f)  \right).
\end{align}
Here $u_{h} \in V_{h}$ is the numerical solution of the method (\ref{IP_method}). 
\end{theorem}

Now we immediately have the following estimates.  Assuming the exact solution $u \in H_{0}^{m}(\Omega) \cap H^s(\Omega)$ for (\ref{m_Laplace_eqs}), $s>m$, we have
\[
\Vert u - u_{h}\Vert_{m, h} \leq C    \left( h^{\min (r + 1 - m, s - m)} \Vert u \Vert_{H^{s}(\Omega)}+ {\rm osc}_m(f)  \right).
\]
In particular, if the oscillation term ${\rm osc}_m(f) $ is zero, we have
\[
\Vert u - u_{h}\Vert_{m, h} \leq C     h^{\min (r + 1 - m, s - m)} \Vert u \Vert_{H^{s}(\Omega)}.
\]

}

\section{Numerical experiments and discussions}

In this section, we provide several numerical experiments to  verify the theoretical prediction of the  $C^{0}$ interior penalty finite element method proposed in the previous sections in two and three dimensions. We calculate the rate of convergence of $\left\|u-u_{h}\right\|_{m,h}$  in various discrete  $H^m$ norms and compare each computed rate with its theoretical estimate. It is pointed out that the estimated convergence rates have very little dependency on the particular value when $\tau=O(1)$, so we choose {$\tau = 1$} in the following tests. { All the numerical experiments are carried out in C, and  the resulting linear algebraic systems are solved using GMRES solvers from the PETSc package \cite{petsc}. }

%In this section, we present numerical experiments to  of   in Section~\ref{sec:analysis} 
%All the numerical experiments are implemented in Matlab.

\subsection*{Example 1.}  For this test, we solve \eqref{method_m_2}, \eqref{method_m_3} and \eqref{method_m_4}, namely $m=2$, 3 and 4, respectively,  using the standard  $r$-th order piecewise continuous $H^1$-conforming finite element space {$V_{h}$} defined in Section~\ref{sec:C0IP} with $\Omega=(0,1)^{2}$. We use the following data:
$$
\begin{aligned}
f(x,y)=&  2^{m} \pi^{2m} \sin (\pi x) \sin (\pi y),
\end{aligned}
$$
so that the exact solution is
$$
u(x,y)=\sin (\pi x) \sin (\pi y),
$$
which satisfies the $m$th-Laplace equation \eqref{m_Laplace_eq1} and homogeneous boundary conditions \eqref{m_Laplace_eq2}.

We list the errors along with their estimated rates of convergence in Tables 1 and 2 when $r=m$ and $r=m+1$, respectively. { It is remarked that long double in C99 standard is used to represent extended precision floating point value for the 4th-order Laplacian operator, which is accurate up to $10^{-20}$.} The tables indicate the following rates of convergence:
$$
\begin{array}{ll}
\left\|u-u_{h}\right\|_{m, h}=O(h),    & \mbox{when} \quad r=m ,\\
\left\|u-u_{h}\right\|_{m, h}=O (h^{2} ),  & \mbox{when} \quad r=m+1.
\end{array}
$$

\begin{table}[bth]
    \centering
    
\caption{Example 1: Errors with estimated rates of convergence when $r=m$ and $m=2$, 3 and 4, respectively. }  \label{tab:ex1a}

\begin{tabular}{c|cc|cc|cc}
\hline $1 / h$ & $\left\|u-u_{h}\right\|_{2, h}$ & Order 
               & $\left\|u-u_{h}\right\|_{3, h}$ & Order 
               & $\left\|u-u_{h}\right\|_{4, h}$ & Order \\
\hline 
8  & $1.1095\mathrm{e}-1$ & $-$    & $4.5054\mathrm{e}-1$ & $-$    & $6.3198\mathrm{e}-1$ & $-$ \\
16 & $5.9870\mathrm{e}-2$ & $0.89$ & $2.4651\mathrm{e}-1$ & $0.87$ & $3.5797\mathrm{e}-1$ & $0.82$ \\
32 & $3.0564\mathrm{e}-2$ & $0.97$ & $1.2672\mathrm{e}-1$ & $0.96$ & $1.9051\mathrm{e}-1$ & $0.91$ \\
64 & $1.5388\mathrm{e}-2$ & $0.99$ & $6.4245\mathrm{e}-2$ & $0.98$ & $9.7934\mathrm{e}-2$ & $0.96$ \\
\hline
\end{tabular}

\end{table}

\begin{table}[bth]
    \centering
    
\caption{Example 1: Errors with estimated rates of convergence when $r=m+1$ and $m=2$, 3 and 4, respectively. }  \label{tab:ex1b}

\begin{tabular}{c|cc|cc|cc}
\hline $1 / h$ & $\left\|u-u_{h}\right\|_{2, h}$ & Order 
               & $\left\|u-u_{h}\right\|_{3, h}$ & Order 
               & $\left\|u-u_{h}\right\|_{4, h}$ & Order \\
\hline 
8  & $2.5510\mathrm{e}-2$ & $-$    & $4.8538\mathrm{e}-2$ & $-$    & $9.0579\mathrm{e}-2$ & $-$ \\
16 & $6.7880\mathrm{e}-3$ & $1.91$ & $1.3371\mathrm{e}-2$ & $1.86$ & $2.5477\mathrm{e}-2$ & $1.83$ \\
32 & $1.7207\mathrm{e}-3$ & $1.98$ & $3.5090\mathrm{e}-3$ & $1.93$ & $6.8738\mathrm{e}-3$ & $1.89$ \\
64 & $4.3317\mathrm{e}-4$ & $1.99$ & $8.9566\mathrm{e}-4$ & $1.97$ & $1.8039\mathrm{e}-3$ & $1.93$ \\
\hline
\end{tabular}

\end{table}

\subsection*{Example 2.}  
In the second example, we test the proposed method in which the solutions have partial regularity on a convex domain \cite{GudiNeilan2011} and a non-convex one \cite{{WuXu1}}, respectively. To this end, we solve the third-Laplace equation 
$$
(-\Delta)^{3} u=f\,.
$$
The first solution is defined on the square domain $\Omega_1=(0,1)^{2}$ with homogeneous Dirichlet boundary conditions. The data $f$ is chosen such that the exact solution is given by
$$
u_1(x,y) = \left(x^{2}+y^{2}\right)^{7.1 / 4}\left(x-x^{2}\right)^{3}\left(y-y^{2}\right)^{3}\,.
$$
Here $u_1\in H^s(\Omega_1)$ and $4\le s<4.1$.

While the second solution is on the $2 \mathrm{D}$ L-shaped domain $\Omega_2 = (-1,1)^{2} \backslash[0,1) \times(-1,0]$  with Dirichlet boundary conditions given explicitly by 
$$
u_2(r,\theta) = r^{2.5} \sin (2.5 \theta)\,.
$$
where $(r, \theta)$ are polar coordinates. Here $f=0$ and $u_2 \in H^{3+1/2}(\Omega_2)$ due to the singularity at the origin. 

In both cases, the observed errors of the proposed method  converge asymptotically  with the optimal order  $h$  and  $h^{1 / 2}$, respectively, in the discrete $H^{3}$ norm, as shown in Table~\ref{tab:ex2}.

\begin{table}[bth]
    \centering
    
\caption{Example 2: Errors with estimated rates of convergence when $r=m=3$. }  \label{tab:ex2}

\begin{tabular}{c|cc|cc}
\hline $1 / h$ & $\left\|u_1-u_{1,h}\right\|_{3, h,\Omega_1}$ & Order 
               & $\left\|u_2-u_{2,h}\right\|_{3, h,\Omega_2}$ & Order \\
\hline 
8  & $7.7412\mathrm{e}-2$ & $-$    & $5.2910\mathrm{e}-1$ & $-$   \\
16 & $4.1198\mathrm{e}-2$ & $0.91$ & $4.1801\mathrm{e}-1$ & $0.34$   \\
32 & $2.1623\mathrm{e}-2$ & $0.93$ & $3.0600\mathrm{e}-2$ & $0.45$   \\
64 & $1.0962\mathrm{e}-2$ & $0.98$ & $2.1488\mathrm{e}-2$ & $0.51$   \\
\hline
\end{tabular}

\end{table}

\subsection*{Example 3.}   
Our last example is a three-dimensional  problem. We take the cubic domain $(0,1)^{3}$ as the computational domain and the exact solution $u$ is given by
$$
u(x,y,z)=\sin (\pi x) \sin (\pi y)  \sin (\pi z),
$$
which satisfies the third-Laplace equation \eqref{m_Laplace_eq1} ($m=3$) and 
homogeneous boundary conditions \eqref{m_Laplace_eq2}.

We list the errors and rates of convergence in Table~\ref{tab:ex3}, 
which indicates that the computed solution converges asymptotically linearly to the exact solution in the discrete $H^3$ norm. The observed rate is in agreement with  Theorem~\ref{thm_conv}.

\begin{table}[bth]
    \centering
    
\caption{Example 3: Errors with estimated rates of convergence when $r=m=3$. }  \label{tab:ex3}

\begin{tabular}{c|cc}
\hline $1 / h$ & $\left\|u-u_{h}\right\|_{3, h}$ & Order 
              \\
\hline 
8  & $3.1836\mathrm{e}-1$ & $-$       \\
16 & $1.6594\mathrm{e}-1$ & $0.94$  \\
32 & $8.5896\mathrm{e}-2$ & $0.95$  \\
64 & $4.3247\mathrm{e}-2$ & $0.99$   \\
\hline
\end{tabular}

\end{table}

\section{Conclusion}
A $C^{0}$ interior penalty method is considered for $m$th-Laplace equation on bounded Lipschitz polyhedral domain in $\mathbb{R}^{d}$ in this paper. In order to avoid computing $D^{m}$ of numerical solution on each element, we reformulate the $C^{0}$ interior penalty method for the odd and even $m$ respectively, and only the gradient and Laplace operators are used in the new method. A rigorous and detailed analysis is given for the key estimate that the discrete $H^{m}$-norm of the solution can be bounded by the natural energy semi-norm associated with our method. Then the stability estimate and the optimal error estimates with respect to discrete $H^{m}$-norm are achieved. {The error estimate under the low regularity assumption of the exact solution is also provided.} We believe that the proposed $C^{0}$ interior penalty method for $m$th-Laplace equation can be applied for the nonlinear high order partial differential equations which will be our consideration in future.

\end{document}